\documentclass[12pt,a4paper]{amsart}

\usepackage[pagewise]{lineno}

\usepackage{amsthm}

\usepackage[all]{xy}

\allowdisplaybreaks

\usepackage[utf8]{inputenc}
\usepackage{bbm,amssymb,mathtools,mathrsfs}
\usepackage[usenames,dvipsnames]{xcolor}

\usepackage[colorlinks,linkcolor=BrickRed,citecolor=OliveGreen,urlcolor=black,hypertexnames=true]{hyperref}
\usepackage[margin=3.5cm]{geometry}

\usepackage{enumerate}

\makeatletter
\renewcommand{\tocsubsection}[3]{%
  \indentlabel{\@ifnotempty{#2}{\hspace*{2.3em}\makebox[2.8em][l]{%
    \ignorespaces#1 #2.\hfill}}}#3}
\makeatother

\numberwithin{equation}{section}

\DeclareMathOperator{\id}{id}
\DeclareMathOperator{\Sym}{Sym}
\DeclareMathOperator{\tr}{tr}

\DeclareMathOperator{\Trd}{Trd}
\DeclareMathOperator{\Real}{Re}

\DeclareMathOperator{\sign}{sign}
\DeclareMathOperator{\Nil}{Nil}

\DeclareMathOperator{\Int}{Int}
\DeclareMathOperator{\Span}{Span}

\DeclareMathOperator{\End}{End}
\DeclareMathOperator{\Char}{char}

\newcommand{\N}{\mathbb{N}}
\newcommand{\Z}{\mathbb{Z}}
\newcommand{\Q}{\mathbb{Q}}
\newcommand{\R}{\mathbb{R}}
\newcommand{\C}{\mathbb{C}}

\newcommand{\CS}{\mathcal{S}}

\newcommand{\CA}{\mathscr{A}}
\newcommand{\CB}{\mathscr{B}}
\newcommand{\CQ}{\mathscr{Q}}
\newcommand{\CM}{\mathscr{M}}
\newcommand{\CN}{\mathscr{N}}
\newcommand{\CP}{\mathscr{P}}
\newcommand{\CU}{\mathscr{U}}

\newcommand{\s}{\sigma}
\newcommand{\ox}{\otimes}
\newcommand{\x}{\times}

\newcommand{\ve}{\varepsilon}
\newcommand{\vt}{\vartheta}

\newcommand{\ovl}{\overline}
\newcommand{\ul}{\underline}

\newcommand{\qf}[1]{\langle #1\rangle}

\newtheorem{thm}{Theorem}[section]
\newtheorem{lemma}[thm]{Lemma}

\newtheorem{cor}[thm]{Corollary}
\newtheorem{prop}[thm]{Proposition}

\theoremstyle{definition}
\newtheorem{defi}[thm]{Definition}

\theoremstyle{remark}
\newtheorem{rem}[thm]{Remark}

\def\m#1{\begin{pmatrix}#1\end{pmatrix}}

\def\of{\text{OF}}

\def\rcf{\text{RCF}}

\def\CSA{\text{CSA}}
\def\CSAI{\text{CSA-I}}

\def\OCSAI{\text{OCSA-I}}
\def\OCSAOI{\text{OCSA-OI}}
\def\OCSASI{\text{OCSA-SI}}
\def\OCSAUI{\text{OCSA-UI}}

\def\HS{\text{HS}}
\def\PSD{\text{PSD}}
\def\NSD{\text{NSD}}

\newcommand{\lF}{\ul{F}}
\newcommand{\ls}{\ul{\s}}
\newcommand{\lCP}{\ul{\CP}}
\newcommand{\lP}{\ul{P}}
\newcommand{\lTrd}{\ul{\Trd}}

\newcommand{\la}{\ul{a}}
\newcommand{\lb}{\ul{b}}

\AddToHook{cmd/label/before}{\phantomsection}

\begin{document}

\title[Quantifier elimination for algebras with involution]{Model completeness and quantifier elimination for (ordered) central simple algebras
with involution}

\author{V. Astier}

\address{School of Mathematics and Statistics,
 University College Dublin, Belfield, Dublin~4, Ireland}


\maketitle

\begin{abstract}
  We show that the theories of some (ordered) central simple algebras with
  involution over real closed fields are model-complete or admit quantifier
  elimination, and characterise positive cones on central simple algebras
  with involution in terms of morphisms into models of some model-complete
  theory. The main result of Section 3 is due to Nicolas Daans.
\end{abstract}

\tableofcontents

\section{Introduction}
In a series of papers (\cite{au14, au15, au18, au20, au22}), the author and T.
Unger considered real algebraic properties of noncommutative algebras with
involution, starting with central simple algebras and the notion of signature of
hermitian forms, which led to considerations linked to sums of hermitian
squares, positive involutions, and a notion of ordering (``positive cones'') on
central simple algebras with involution.

In the commutative case, orderings on fields correspond to inclusions into
well-behaved ordered fields: the real closed fields. In model-theoretic
terms, the theory of real closed fields admits quantifier elimination, is
model-complete, and can be characterised as the model-completion of the theory of
ordered fields.

It is therefore natural to investigate if some theory of ordered (i.e., equipped
with a positive cone) central simple algebras with involution would be model-complete, eliminate quantifiers,
or be the model-completion of some theory of ordered central simple algebras
with involution.
The models of this theory would then be candidates for being the well-behaved
ordered central simple algebras with involution.

This is what we do in this paper. After some preliminaries, we present an
axiomatization of central simple algebras (with or without involution) of
fixed dimension, and show in Section 3 that we obtain model-complete
theories when we specify that their centres (or the field of symmetric
elements of their centres in case of algebras with involution) are models of
a model-complete theory $T$ (Theorem~\ref{mod-comp-ec}). The reference
\cite{ace} (underlying the proof of Lemma~\ref{csa-centre}) was pointed out
by Nicolas Garrel.

We then consider in Section 4 the questions of model-completeness and quantifier
elimination for ordered central simple algebras with involution, and first prove that the
theories of such algebras admit quantifier elimination in simple cases if $T$ is the
theory of real closed fields (Proposition~\ref{qe-mat}). This is obtained by very closely following
the arguments used in \cite{kt25}. The rest of section 4 is devoted to ordered central
simple algebras: We recall the relevant definitions and results in Section 4.1,
then prove the model-completeness of these theories when $T$ is the theory of
real closed fields in Section 4.2 (Proposition~\ref{mc-ocsai}), and some corresponding quantifier elimination
results in Section 4.3 (Theorem \ref{KT-eq}; again closely following the
arguments used in \cite{kt25}).

Finally, Section 5 shows that, similarly to the commutative case, positive cones
on central simple algebras with involution (of a given dimension) correspond to morphisms into models
of the model-completion of the theory of ordered central simple algebras with
involution of the same dimension (Propositions~\ref{one-direction} and
\ref{other-direction}).

Concerning the very much related question of the description of definable sets,
the recent paper of Savi (\cite{savi26}) is worth noting. It shows the model completeness,
in the language of rings, of a natural theory of quaternions (and octonions),
described in terms of quaternionic zeroes of polynomials.  The paper then
investigates the properties and descriptions of such zero sets of polynomials
and of algebraic sets (defined in a precise way), i.e., of many of the
quantifier-free definable sets considered in \cite{savi26} and in the present paper.

\section{Preliminaries}

\subsection{Algebras with involution}\label{awi}
All fields in this paper will have characteristic different from $2$.
Our main reference for central simple algebras with involutions is \cite{BOI},
and we simply recall what will be needed in the paper.

By central simple algebra over a field $K$ we mean a simple $K$-algebra $A$ with 1 that is
finite-dimensional over $K$ and such that $K$ is equal to $Z(A)$, the centre of
$A$. Such an algebra is
isomorphic to a matrix algebra $M_\ell(D)$, for a unique $\ell \in \N$ and a
$K$-division algebra $D$ that is unique up to $K$-isomorphism. A splitting field $L$
of $A$ is an extension $L$ of $K$ such that $A \ox_K L \cong M_m(L)$ for some
$m$. Such a splitting field always exists (for instance the algebraic closure of
$K$), and $\deg A := m = \sqrt{\dim_{K} A}$ is called the degree of
$A$.\medskip

If $A$ is a ring and $\s$ is an involution on $A$, we denote by
\[\Sym(A,\s) := \{a \in A \mid \s(a) = a\}\]
the set of symmetric elements of $A$.\medskip

In this paper, $F$ will always denote a field. By a central simple algebra with
involution over $F$ we mean a pair $(A,\s)$ where $A$ is a finite-dimensional
simple $F$-algebra with 1, and where $\s$ is an
involution on $A$ such that $F = Z(A) \cap \Sym(A,\s)$.  Note that $\s$ is then
$F$-linear, and that $[Z(A):F] \le 2$. We call $F$ the base field of $(A,\s)$.

For $a \in A^\x$, we denote by $\Int(a)$ the inner automorphism determined
by $a$:
\[\Int(a): A \rightarrow A,\ x \mapsto axa^{-1}.\]

The involution $\s$ is said to be of the first kind if $F = Z(A)$, and of the
second kind if $[Z(A):F] = 2$. A finer classification of involutions is given by
their type:

Involutions of the first kind can have two types, described as follows (with
the notation $m := \deg A$): orthogonal if $\dim_F
\Sym(A,\s) = \dfrac{m(m+1)}{2}$, or symplectic if $\dim_F \Sym(A,\s) =
\dfrac{m(m-1)}{2})$, cf. \cite[Proposition~2.6]{BOI}. Involutions of the second
kind are also called of unitary type.\medskip

Recall that, by the Skolem-Noether theorem:
\begin{prop}[{\cite[Propositions~2.7 and 2.18]{BOI}}]\label{SK-invo}
  If $\s$ and $\gamma$ are two $F$-linear involutions on $A$ and are of the same
  kind, then there is $a \in A^\x$ such that $\s = \Int(a) \circ \gamma$.
\end{prop}

We will be particularly interested in central simple algebras with involution
whose base field $F$ is real closed. As recalled above, they are (up to
isomorphism) of the form $M_\ell(D)$ where $D$ is a finite-dimensional division
algebra over $F$. Since $F$ is real closed $D$ is one of $F$, $F(\sqrt{-1})$, or
$(-1,-1)_F$ (where $(-1,-1)_F$ denotes the quaternion algebra over $F$ with
usual basis $\{1,i,j,k\}$ such that $i^2=j^2=-1$ and $ij=-ji=k$). We will denote
the canonical $F$-linear involutions on $F(\sqrt{-1})$ and $(-1,-1)_F$ by $-$ in
both cases.\medskip

\begin{rem}\label{F-div-alg}
  Let $F$ be a field.
  \begin{itemize}
    \item By $F(\sqrt{-1})$ we mean $F[X]/(X^2+1)$, so that $F(\sqrt{-1})$
      is a field if and only if $-1$ is not a square in $F$, 
      if and only if $a^2+b^2=0$ implies $a=b=0$ for
      every $a,b \in F$, if and only if the quadratic form $\qf{1,1}$ is
      anisotropic over $F$.
    \item The standard quaternion algebra $(-1,-1)_F$ is a division algebra
      if and only if $a^2+b^2+c^2+d^2=0$ implies $a=b=c=d=0$ for every
      $a,b,c,d \in F$, if and only if the quadratic form $\qf{1,1,1,1}$ is
      anisotropic, if and only if $-1$ is not a sum of two squares in $F$.
  \end{itemize}
\end{rem}

Let $F$ be a field such that $(D, \vt) \in \{(F, \id), (F(\sqrt{-1}),-),
((-1,-1)_F,-)\}$ is a division algebra with involution. Let $n \in \N$.
The involution $\vt^t$ on $M_n(D)$ is defined by
\[\vt^t((d_{ij})_{i,j}) := (\vt(d_{ji}))_{i,j},\]
and is
\[\begin{cases}
  \text{orthogonal} & \text{if } (D, \vt) = (F, \id) \\
  \text{symplectic} & \text{if } (D, \vt) = ((-1,-1)_F, -) \\
  \text{unitary} & \text{if } (D, \vt) = (F(\sqrt{-1}), -).
\end{cases}\]
In any of these three situations, and if $P$ is an ordering on $F$,
\[\PSD_P(M_n(D), \vt^t) := \{B \in \Sym(M_n(D), \vt^t) \mid \forall X \in D^n
\quad \vt(X)^t B X \in P\}\]
will denote the set of symmetric matrices that are positive-semidefinite at $P$. We
will often drop the reference to $P$ or to $(M_n(D), \vt^t)$ if they are clear
from the context.

\subsection{Model theory}
In this paper, ``language'', ``formula'' and ``theory'' will mean first-order
language, formula and theory, respectively.
We fix a language $L$ for the remainder of this section.

We will use the standard notation for $L$-structures (see for instance
\cite[p~8]{marker02}), but will not distinguish between a structure and its base
set: If $\CM$ is an $L$-structure, we will denote the base set of $\CM$ also by
$\CM$.
If $S$ is a symbol in $L$ and $\CM$ is an $L$-structure, $S^\CM$ will
denote the interpretation of $S$ in $\CM$.\medskip
\begin{defi}
  Let $\CM$ and $\CN$ be $L$-structures, and let $f : \CM \rightarrow \CN$
  be an $L$-embedding. We say that $f$ is
  \begin{enumerate}
    \item existentially closed if, for every existential $L$-formula
      $\phi(\bar a)$ with parameters $\bar a \in \CM$, we have $\CM \models
      \phi(\bar a)$ if and only if $\CN \models \phi(f(\bar a))$.
    \item an elementary embedding if, for every $L$-formula
      $\phi(\bar a)$ with parameters $\bar a \in \CM$, we have $\CM \models
      \phi(\bar a)$ if and only if $\CN \models \phi(f(\bar a))$. We use
      the usual notation $\prec$ to indicate elementary embeddings:
      \[\xymatrix{f : \CM \ar[r]^-\prec & \CN.}\]
  \end{enumerate}
\end{defi}

We introduce the notation that we will use for ultrapowers, as well as one of
their applications.

If $I$ is an index set, $\CM_i$ is an $L$-structure (for $i \in I$) and $\CU$ is
an ultrafilter on $I$, we denote by $\Pi_{i \in I} \CM_i/\CU$ the ultraproduct
of the $\CM_i$ modulo $\CU$. Its definition can be found in
\cite[Exercise~2.5.19]{marker02} or \cite[p.~215]{chang-keisler}.

If all the $\CM_i$ are equal to a given $L$-structure $\CM$, then
\[\CM^I/\CU = \Pi_{i \in I} \CM_i/\CU\]
is called an ultrapower of $\CM$, and the canonical $L$-embedding:
\[\delta : \CM \rightarrow \CM^I/\CU,\quad a \mapsto (a)_{i \in
I}/\CU\]
is elementary.

We recall \cite[Theorem~6.4]{eklof}:

\begin{prop}\label{ec-ultra}
  Let $\CM$ and $\CN$ be $L$-structures and let $f : \CM \rightarrow \CN$ be an $L$-embedding. The following are
  equivalent:
  \begin{enumerate}
    \item $f$ is existentially closed;
    \item There is a set $I$, an ultrafilter $\CU$ on $I$, and an $L$-embedding
      $\eta$ such that the following diagram is commutative
      \[\xymatrix{
        \CM \ar[r]^f \ar@/_1.5pc/[rr]_\delta & \CN \ar[r]^-\eta & \CM^I/\CU,
      }\]
      where $\delta$ is the canonical inclusion of $\CM$ in $\CM^I/\CU$.

  \end{enumerate}
\end{prop}

We will work with algebras with involution, and will be interested in various
maps and relations that are naturally considered in this context, for instance
the involution (often denoted $\s$), the reduced trace map (denoted $\Trd$),
etc. Our languages will contain symbols that will be interpreted by such an
involution, field, reduced trace map, etc. In order to make them easily
recognizable, we will use the same symbols in the language, but underlined (so:
$\ls$, $\lTrd$, etc).

Frequently, an algebra with involution $(A,\s)$ will be model of a theory
whose language contains more symbols than just $\s$. We will still denote it by
$(A,\s)$, or even by $A$, when no confusion will seem likely to arise.

\subsection{Axiomatization of finite-dimensional central simple algebras}
\mbox{}
Recall that a ring $S$ is called von Neumann regular if for every $x \in S$
there is $y \in S$ such that $x=xyx$. The following result is well-known.

\begin{lemma}\label{ax-csa}
  Let $F$ be a field and let $A$ be a finite-dimensional $F$-algebra. Then
  \begin{enumerate}
    \item $A$ is semisimple if and only if $A$ is von Neumann regular.
    \item Assume that $Z(A)$ is a field. Then $A$ is a central simple
      algebra over its centre if and only if $A$ is von Neumann regular.
  \end{enumerate}
\end{lemma}
\begin{proof}
  Note that $A$ is left and right Noetherian since $\dim_F A < \infty$.
  \begin{enumerate}
    \item The left to right implication is \cite[Corollary~4.24]{lam91},
      while the right to left one is \cite[Theorem~4.25]{lam91}.
    \item ``$\Rightarrow$'' We know that $A \cong M_n(D)$ for some $n \in
      \N$ and some $Z(A)$-division algebra $D$. Then $A$ is semisimple, and
      the first part of the lemma applies.\\
      ``$\Leftarrow$'' By the first part, $A$ is semisimple, so is a finite product of
      matrix rings over division rings (\cite[Theorem~3.5]{lam91}). Since
      its centre is a field and thus not a product of more than one field,
      $A$ is a single matrix ring over a division ring, so is central
      simple.\qedhere
  \end{enumerate}
\end{proof}

For $m \in \N$ we define the following two theories:
\begin{align*}
  \CSA_m :=&\text{ the theory of von Neumann regular rings} \\
  &\cup \{\text{the centre is a field}\}\\
  &\cup \{\text{the dimension over the centre is } m\},
\end{align*}
in the language $L_R$ of rings and, in the language $L_\CSAI := L_R \cup
\{\ls\}$: 
\begin{align*}
  \CSAI_m :=& \text{ the theory of von Neumann regular rings } \\
    & \cup \{\text{the centre is a field, $\ls$ is an involution}\} \\
    & \cup \{\text{the dimension over the base field is $m$}\},
\end{align*}
recalling that the base field of a central simple algebra with involution
$(A,\s)$ is $Z(A) \cap \Sym(A,\s)$
($\CSA$ stands for central simple algebra and $\CSAI$ for central simple algebra
with involution).
Lemma~\ref{ax-csa} immediately gives:
\begin{cor}\label{axiom-csa}
  \begin{enumerate}
    \item The models of $\CSA_m$ are exactly the central simple algebras of
      dimension $m$ over their centres.
    \item The models of $\CSAI_m$ are exactly the central simple algebras with
      involution of dimension $m$ over their base fields.
  \end{enumerate}
\end{cor}

\begin{rem}
  In the two theories $\CSA_m$ and $\CSAI_m$, the theory of von Neumann regular
  rings (used to obtain the simplicity of the algebra) can be replaced by saying
  that the 2-sided ideal generated by any non-zero element contains 1. Since the
  dimension of the algebra over its centre is $m \in \N$, this can be expressed
  by the following first-order formula:
  \[\forall a \not = 0\ \exists x_1, y_1, \ldots, x_m, y_m \quad 1=\sum_{i=1}^m
  x_iay_i.\]
  It gives the same models, but it introduces one more reference to the
  finiteness of the dimension, which is more restrictive if one day this theory
  can be extended to some infinite-dimensional algebras.
\end{rem}

\subsection{The reduced trace, the $\star$ operation, and words of matrices}

We will consider two different traces.
\begin{enumerate}
  \item The usual matrix trace on $M_n(D)$, where $D$ is a division algebra:
    \[\tr((a_{ij})_{i,j=1, \ldots, n}) = \sum_{i=1}^n a_{ij}.\]
  \item The reduced trace, $\Trd_A$, where $A$ is a central simple algebra
    (see \cite[Section~1A]{BOI}; we will often simply write $\Trd$ instead of
    $\Trd_A$ if the algebra is clear from the context).

    If $A$ is a central simple algebra over $K := Z(A)$, the reduced trace is a
    $K$-linear map from $A$ to $K$.
    It is obtained by extending the scalars to a splitting field $L$ of $A$,
    (i.e., $K \subseteq L$ and there is an isomorphism of $L$-algebras $f:
    A \ox_{K} L \rightarrow M_m(L)$ where $m=\deg A$), and then by taking the
    usual trace in $M_m(L)$. It can be shown that the result does not depend on the
    choice of $L$ or of the isomorphism $f$ (in particular the reduced trace is
    invariant under isomorphisms of $F$-algebras). So, for $a \in A$:
    \[\Trd_A(a) := \tr(f(a \ox 1)).\]
\end{enumerate}

\begin{rem}\label{equal-field}
  \begin{enumerate}
    \item These two traces produce different results for a central simple
      algebra of the form $M_\ell(D)$ where $D$ is a division algebra that
      is not a field (the reduced trace will have values in $Z(M_\ell(D)) =
      Z(D)$), but are equal if $D$ is a field since there is no need to
      extend scalars to split the algebra.
    \item Observe that, if $A$ is a central simple algebra and $\Char Z(A)$
      does not divide $\deg A$, then $a \in Z(A)$ if and only if $\Trd(a) =
      (\deg A)\cdot a$ for every $a \in A$.
  \end{enumerate}
\end{rem}

\begin{lemma}\label{axiom-Trd}
  Let $A$ be a central simple algebra over $K$ and let $f : A \rightarrow K$ be
  $K$-linear such that $f(xy) = f(yx)$ for every $x,y \in A$ and $f(1) = \deg
  A$. Then $f = \Trd_A$.
\end{lemma}
\begin{proof}
  Let $L$ be a splitting field of $A$, and $\omega: A \ox_K L \rightarrow
  M_m(L)$ an isomorphism of $L$-algebras. A direct verification shows that $(f \ox
  \id)(xy) = (f \ox \id)(yx)$ for every $x,y \in A \ox_K L$ (for
  some $m \in \N$), and we still have $(f \ox \id)(1) = \deg A = \deg M_m(L)$.
  Therefore, by definition of the reduced trace, it suffices to show that
  $(f \ox \id) \circ \omega^{-1}$ is the reduced trace on $M_m(L)$, i.e., it suffices to show the
  result for $A = M_m(L)$.

  For $r,s \in \{1, \ldots, m\}$, let $E_{rs}$ be the matrix with $1$ at entry
  $(r,s)$ and $0$ elsewhere. We have, for $r \not = s$:
  \[f(E_{rs}) = f(E_{rr}E_{rs}) = f(E_{rs}E_{rr}) = f(0) = 0.\]
  Furthermore, $f(XYX^{-1}) = f(Y)$ by hypothesis, for every $Y \in M_m(L)$ and
  $X \in M_m(L)^\x$. Therefore, if $P_{rs}$ is the permutation matrix corresponding to the
  transposition $(r\ s)$, we have $P_{rs}E_{rr}{P_{rs}}^{-1} = E_{ss}$, so that
  \[f(E_{ss}) = f(P_{rs}E_{rr}{P_{rs}}^{-1}) = f(E_{rr}).\]
  In particular $f(I_m) = \sum_{r=1}^m E_{rr} = mE_{ss}$ for any $s \in \{1,
  \ldots, m\}$.

  Since $f(I_m) = m = \sum_{r=1}^m f(E_{rr})$, it follows that $f(E_{rr}) = 1$
  for every $r$, proving that $f = \Trd$ on $\{E_{rs}\}_{r,s=1, \ldots, m}$,
  which is a basis of $M_m(L)$ over $L$.
\end{proof}

\begin{lemma}\label{trace-morphism}
  Let $A$ and $B$ be two $F$-algebras such that $A$ is central simple over
  $Z(A)$ and there is an isomorphism of $F$-algebras $f : A \rightarrow B$.
  Then, for every $a \in A$, $f(\Trd_A(a)) = \Trd_B(f(a))$.
\end{lemma}
\begin{proof}
  Let $f_0 := f|_{Z(A)} : Z(A) \rightarrow Z(B)$. Using the action of $Z(A)$
  on $Z(B)$ induced by $f_0$, we can build the map $g : A' := A \ox_{Z(A)} Z(B)
  \rightarrow B$, $a \ox z \mapsto f(a) z$, which is an isomorphism of
  $Z(B)$-algebras (it is injective since $A'$ is simple, and $\dim_{Z(B)} A' =
  \dim_{Z(A)} A = \dim_{Z(B)} B$). In particular, if $a \in A$ then $\Trd_{A'}(a
  \ox 1) = \Trd_B(f(a))$.

  We now consider the morphism of rings $f_0 : Z(A) \rightarrow Z(B)$. By
  \cite[Theorem~4.3 e)]{saltman99} we have $f_0(\Trd_A (a)) = \Trd_{A'}(a \ox
  1)$. The result follows.
\end{proof}

We consider more closely the case of the reduced trace on $M_n(D)$, where $D =
(-1,-1)_F$ is a quaternion division algebra over $F$. A splitting extension of
$D$ is given by $L:=F(\sqrt{-1})$, and the map
\[f_0: D \ox_F L \rightarrow M_2(L),\]
\[i \ox 1 \mapsto \m{0&1\\-1&0},\  
    j \ox 1 \mapsto \m{0&\sqrt{-1}\\\sqrt{-1}&0}, \ 
    k \ox 1 \mapsto \m{\sqrt{-1}&0\\0&-\sqrt{-1}}\]
is an isomorphism of $L$-algebras. It induces a morphism $f: M_\ell((-1,-1)_L)
\rightarrow M_{2\ell}(L)$ such that, if $a = (a_{rs})$, then $f(a) =
(f_0(a_{rs}))$. Therefore,
\[\Trd(a) = \sum_{r=1}^\ell \tr(f_0(a_{rr})).\]

Writing $a_{rr} = u_{r,1} + iu_{r,2} + ju_{r,3} + ku_{r,4}$ with $u_{r,1},
\ldots, u_{r,4} \in F$, we have
\begin{align*}
  \tr(f_0(a_{rr})) &= u_{r,1} \tr(I_2) + u_{r,2}\tr(f(i)) + u_{r,3}\tr(f(j)) +
  u_{r,4}\tr(f(k)) \\
    &= 2u_{r,1}.
\end{align*}
Thus:
\begin{equation}\label{reduce-real-trace}
  \Trd(a) = 2 \sum_{i=1}^n u_{r,1} = 2 \Real(\tr(a)) = 2\tr(\Real(a)).
\end{equation}\medskip

We now introduce the $*$ operation, following \cite{lee49} and \cite{wiegmann62}
(we use Wiegmann's notation $*$ from \cite{wiegmann62} since we will mostly
refer to this paper; Lee denotes it by the function $f$): Let $F$ be a field
such that $(-1,-1)_F$ is a division algebra, and let $M \in M_n((-1,-1)_F)$,
written as $M = M_1 + jM_2$ where $M_1, M_2 \in M_n(F(\sqrt{-1}))$. We define:
\[M^* := \m{M_1 & -\ovl{M_2} \\ M_2 & \ovl{M_1}}\]
From \cite[Section~4]{lee49}, we have
\begin{prop}\label{*-morphism}
  The map $X \mapsto X^*$ is an injective morphism of rings with involution from
  $(M_n((-1,-1)_F), -^t)$ to $(M_{2n}(F(\sqrt{-1})), -^t)$.
\end{prop}

Considering $M \in M_n((-1,-1)_F)$ written as $M = M_1 + jM_2$ with $M_1, M_2 \in
M_n(F(\sqrt{-1}))$, a direct computation shows that
\[\tr(M^*) = \tr(M_1) + \tr(\ovl{M_1}) = 2\tr(\Real(M)).\]
Putting this together with \eqref{reduce-real-trace}, we obtain:
\begin{lemma}\label{Trd-tr}
  Let $F$ be a field such that $(-1,-1)_F$ is a division algebra, and let $M \in
  M_n((-1,-1)_F)$. Then
  \[\Trd(M) = \tr(M^*).\]
\end{lemma}

We use these observations to reformulate in a uniform way some results from
several authors (see the proof for the references) on unitary similarity of
tuples of matrices. These results are already recalled in
\cite[Theorem~2.2.2]{kt25} for the real and complex cases.

\begin{thm}\label{Specht-property}
  Let $(D,\vt) \in \{(F,\id), (F(\sqrt{-1}),-), ((-1,-1)_F, -)\}$ where $F$ is a
  real closed field, and let $d \in \N$. Then the following are equivalent, for
  any $X, Y  \in M_n(D)^d$:
  \begin{enumerate}
    \item\label{SP3} There is $O \in M_n(D)$ such that $\vt(O)^tO = I_n$ and
      $\vt(O)^tX_iO = Y_i$ for $i=1, \ldots, d$.
    \item\label{SP1} For every word $w$ in $x_1, \ldots, x_d, \vt(x_1)^t, \ldots,
      \vt(x_d)^t$, the following holds: $\Trd(w(X,\vt(X)^t)) = \Trd(w(Y,\vt(Y)^t))$.
    \item\label{SP2} For every word $w$ in $x_1, \ldots, x_d, \vt(x_1)^t,
      \ldots, \vt(x_d)^t$ of length at most $n^2$, the following holds:
      $\Trd(w(X,\vt(X)^t)) = \Trd(w(Y,\vt(Y)^t))$.
  \end{enumerate}
\end{thm}
\begin{proof}
  We first assume that $F = \R$. Recall from Remark~\ref{equal-field} that
  $\Trd_{M_n(D)} = \tr$ when $D = \R$ or $\C$. 
  For $D = \R$, the equivalence of \eqref{SP3}
  and \eqref{SP1} is \cite[Lemma~2]{sibirskii68}. For $D = \C$, the equivalence of \eqref{SP3}
  and \eqref{SP1} follows from \cite[Theorem~4]{wiegmann62} with one extra 
  verification for \eqref{SP1} implies \eqref{SP3}: The map in the hypothesis of Wiegmann's result is given by
  \[w(X, \vt(X)^t) \mapsto w(Y,\vt(Y)^t),\]
  for all words $w$ in $x_1, \ldots, x_d, \vt(x_1)^t, \ldots, \vt(x_d)^t$.
  It satisfies the required properties, but we need to check that it is
  well-defined, which we now do. Let $w_1$ and $w_2$ be two words in $x_1,
  \ldots, x_d, \vt(x_1)^t, \ldots, \vt(x_d)^t$. We have (using that a matrix $B
  \in M_n(D)$ is zero if and only if $B \cdot \vt(B)^t$ is zero):
  \begin{align*}
    &w_1(X,\vt(X)^t) = w_2(X,\vt(X)^t) \Leftrightarrow w_1(X,\vt(X)^t) -
        w_2(X,\vt(X)^t) = 0 \\
      &\Leftrightarrow \left(w_1(X,\vt(X)^t) - w_2(X,\vt(X)^t)\right)\cdot
      \vt\left(w_1(X,\vt(X)^t) -
        w_2(X,\vt(X)^t)\right)^t = 0 \\
      &\Leftrightarrow \tr\big[\left(w_1(X,\vt(X)^t) - w_2(X,\vt(X)^t)\right)\\
      &  \strut\hspace{.4\linewidth} \cdot \vt\left(w_1(X,\vt(X)^t) -
        w_2(X,\vt(X)^t)\right)^t\big] = 0 \\
      & \qquad \text{since, for a PSD matrix $B$, $B = 0$ if and only if $\tr B =
      0$} \\
      &\Leftrightarrow \sum_{i,j=1}^2 (-1)^{i+j}\tr\big[w_i(X,\vt(X)^t) \cdot
      \vt\left(w_j(X,\vt(X)^t)\right)^t\big] = 0 \\
      &\Leftrightarrow \sum_{i,j=1}^2 (-1)^{i+j}\tr\big[w_i(Y,\vt(Y)^t) \cdot
      \vt\left(w_j(Y,\vt(Y)^t)\right)^t\big] = 0 \\
      &\qquad \text{by hypothesis} \\
      &\Leftrightarrow \tr\big[\left(w_1(Y,\vt(Y)^t) -
      w_2(Y,\vt(Y)^t)\right)\\
      &  \strut\hspace{.4\linewidth}   \cdot \vt\left(w_1(Y,\vt(Y)^t) -
        w_2(Y,\vt(Y)^t)\right)^t\big] = 0 \\
      &\Leftrightarrow \left(w_1(Y,\vt(Y)^t) - w_2(Y,\vt(Y)^t)\right)\cdot
      \vt\left(w_1(Y,\vt(Y)^t) -
        w_2(Y,\vt(Y)^t)\right)^t = 0 \\
      &\Leftrightarrow w_1(Y,\vt(Y)^t) = w_2(Y,\vt(Y)^t).
  \end{align*}
  The degree bounds for $D=\R$ and
  $D=\C$ are from \cite[Theorem~7.3]{procesi76}, \cite[Razmyslov's Theorem,
  p.~451]{procesi07}.

  If $(D, \vt) = ((-1,-1)_\R, -)$, the result is a consequence of 
  \cite[Theorem~4]{wiegmann62} together with \cite[Theorem~1]{wiegmann62}, as
  briefly explained in the final paragraph on the same
  page as \cite[Theorem~4]{wiegmann62}, since $\Trd = 2 (\tr \circ \Real)$ in
  this case. We give some details since not many are given, and in order to
  make it clear that \eqref{SP2} is also covered:

  \eqref{SP3}$\Rightarrow$\eqref{SP1}: From $\vt(O)^tX_iO = Y_i$ we get
  $\vt(O)^t\vt(X_i)^tO = \vt(Y_i)^t$ for every $i$. Therefore we have
  $\vt(O)^t w(X,
  \vt(X)^t) O = w(\vt(O)^tXO, \vt(O)^t\vt(X)^tO) = w(Y,\vt(Y)^t)$, and the
  result follows since the reduced trace is invariant under $F$-algebra
  isomorphisms.

  \eqref{SP1}$\Rightarrow$\eqref{SP2} is clear.

  \eqref{SP2}$\Rightarrow$\eqref{SP3}: By Lemma~\ref{Trd-tr} we have
  $\tr(w(X,\vt(X)^t)^*) = \tr(w(Y,\vt(Y)^t)^*)$ for every word $w$ of length at
  most $n^2$. Since $*$ is a morphism of algebras with involution (cf.
  Proposition~\ref{*-morphism}) we get $\tr(w(X^*,\vt(X^*)^t)) = \tr(w(Y^*,
  \vt(Y^*)^t))$, i.e., $\Trd(w(X^*,\vt(X^*)^t)) = \Trd(w(Y^*, \vt(Y^*)^t))$
  since we are taking the trace of complex matrices, cf.
  Remark~\ref{equal-field}. By the complex case, we get
  a unitary matrix $U$ such that $\vt(U)^t {X_i}^* U = {Y_i}^*$ for every $i =
  1, \ldots, d$. By \cite[Theorem~1]{wiegmann62} there is a unitary matrix $V$
  in $M_{2n}(\C)$ such that $V=O^*$ for some $O \in M_n((-1,-1)_\R)$ and
  $\vt(V)^t {X_i}^* V = {Y_i}^*$ for every $i = 1, \ldots, d$ (the proof of
  \cite[Theorem~1]{wiegmann62} shows that $V$ only depends on $U$, so is the
  same for every $i$), i.e., $\vt(O^*)^t {X_i}^* O^*= {Y_i}^*$ for every $i =
  1, \ldots, d$. It is direct to check that $O$ is unitary and, using that $*$
  is an injective morphism of algebras with involution, we obtain first
  $(\vt(O)^t X_i O)^* = Y_i^*$ and then $\vt(O)^t X_i O = Y_i$ for every $i =
  1, \ldots, d$.

  For the general case where $F$ is a real closed field, observe that the
  equivalence of \eqref{SP3} and \eqref{SP2} can be expressed as a first-order
  formula in the language of fields, and thus follows from the case $F = \R$.
  And it is clear that \eqref{SP3} $\Rightarrow$ \eqref{SP1} (with the same
  proof as above) and that \eqref{SP1} $\Rightarrow$ \eqref{SP2}.
\end{proof}

\begin{rem}
  We need to use the reduced trace instead of the usual trace for
  matrices with quaternion coefficients, since there are matrices $A,B,U \in
  M_2((-1,-1)_\R)$ such that $\ovl{U}^t U = I_2$, $B = \ovl{U}^t A U$ but
  $\tr(B) \not = \tr(A)$, see \cite[Example 7.2]{zhang97}.
\end{rem}

\section{Model completeness}
The main result of this Section, Theorem~\ref{mod-comp-ec}, describes model-complete
theories of central simple algebras with or without involution. Its
current form, for a model-complete theory T, as well as the strategy for
the proof (Lemmas~\ref{csa-centre} and \ref{ec}), are due to Nicolas Daans. The author
wishes to express his thanks for this valuable contribution which
substantially improved the paper.

The following Lemma is a special case of \cite[Theorem~4.1(2)]{ace}, which is
stated more generally for Azumaya algebras. Since \cite{ace} seems to have only
appeared as a preprint, we present its second proof, given in \cite[p.~10]{ace},
simplified to the case of central simple algebras.
\begin{lemma}\label{csa-centre}
  Let $A$ and $B$ be central simple algebras, such that $\dim_{Z(A)} A =
  \dim_{Z(B)} B$.
  Let $\phi : A \rightarrow B$ be a morphism of rings. Then
  \begin{enumerate}
    \item $\phi(Z(A)) \subseteq Z(B)$;
    \item the map $\phi$ sends any $Z(A)$-basis of $A$ to a $Z(B)$-basis of $B$.
  \end{enumerate}
\end{lemma}
\begin{proof}
  Recall that if $\psi : C \rightarrow C'$ is a (necessarily injective) morphism of rings between two
  central simple algebras, of respective dimensions $n^2$ and $n'^2$ over their
  centres, then $n \le n'$. It follows from the Artin-Procesi theorem
  \cite[Theorem~6.1.35]{rowen88}, since both are Azumaya algebras: $C$
  satisfies all multilinear identities of $M_n(\Z)$ and there is some
  multilinear identity $f$ of $M_{n-1}(\Z)$ that is not satisfied in $\psi(C)$.
  Assume that $n'<n$. Since $\psi(C) \subseteq C'$, we know that $\psi(C)$ satisfies all multilinear identities of
  $M_{n'}(\Z)$ (since $C'$ does), and in particular $f$ since $n' \le n-1$ (using
  \cite[Lemma~6.1.27']{rowen88}: $f$ is an identity of $M_{n'}(\Z)$),
  contradiction.
  \begin{enumerate}
    \item  We show that $Z(A) \subseteq Z(B)$. For simplicity of notation, let
      $F := Z(A)$ and $L := Z(B)$. Let $k$ be the algebraic closure
      of $L$. We have an isomorphism $g : B \ox_L k \cong M_n(k)$ for some
      $n \in \N$, and we thus have
      \[\xymatrix{
        A \ar[r]^-\phi & B \ar@{^{(}->}[r]^-\iota & B \ox_L k \ar[r]^-g & M_n(k).
      }\]
      Since $g \circ \iota$ is injective, we have $B \cap Z(M_n(k)) \subseteq Z(B)$
      (considering $g \circ \iota$ as an inclusion). In particular, we may assume
      that $B = M_n(k)$ and $L = k$, and the map $\phi$ is thus $\phi : A \rightarrow
      M_n(k)$.

      Let $c \in Z(A)$, let $\lambda$ be an eigenvalue of $\phi(c)$, and let
      \[V_\lambda := \{v \in k^n \mid \phi(c)v = \lambda v\}.\]
      Since $c \in Z(A)$, we have $\phi(c)\phi(x) = \phi(x)\phi(c)$ for every $x
      \in A$, and therefore $\phi(x)V_\lambda \subseteq V_\lambda$. We thus get
      a morphism of rings
      \[A \rightarrow \End_k(V_\lambda) \cong M_{n'}(k),\ x \mapsto
      \phi(x)|_{V_\lambda},\]
      where $n' = \dim_k V_\lambda \le n$.  By the dimension inequality that we
      recalled in the first paragraph of the proof, we obtain $n' \ge n$, so
      $n'=n$ and $V_\lambda = k^n$. It follows that $\phi(c) = \lambda I_n \in
      Z(M_n(k))$.

    \item Since $\phi(F) \subseteq L$, we have $\id \ox \phi : A \ox_\phi L
      \rightarrow B$. This map is a morphism of $L$-algebras, which is injective
      since $A \ox_\phi L$ is simple, and surjective since $\dim_L (A
      \ox_\phi L) = \dim_F A = \dim_L B$.  The result follows. \qedhere
  \end{enumerate}
\end{proof}

Let $T$ be a theory of fields in a language $L_T$ containing the language
of rings $L_R$ (we also assume that $L_T \cap L_\CSAI = L_R$, i.e., that
$\ls \not \in L_T$). We define
\[\CSA_{m,T} := \CSA_m \cup \{\text{the centre is a model of $T$}\} \cup
\Sigma\]
in the language $L_T$, and
\[\CSAI_{m,T} := \CSAI_m \cup \{
  \text{the base field is a model of $T$}\} \cup \Sigma\]
in the language $L_\CSAI \cup L_T$, and where $\Sigma$ is a set of formulas
describing how to interpret the symbols from $L_T \setminus L_R$ in a model of
$\CSA_m$ or $\CSAI_m$, which we do as follows:
\begin{itemize}
  \item There is nothing to specify for the constant symbols;
  \item For every relation symbol $S$ in $L_T \setminus L_R$, $\Sigma$
    contains a formula indicating that $S(x_1, \ldots, x_n)$ does not hold
    if one of the $x_i$ does not belong to: the centre in the case of
    $\CSA_{m,T}$, the base field in the case of $\CSAI_{m,T}$.
  \item For every function symbol $f$ in $L_T \setminus L_R$, $\Sigma$ contains a formula
    indicating that $f(x_1, \ldots, x_n) = 0$ if one of the $x_i$
    does not belong to: the centre in the case of $\CSA_{m,T}$, the base
    field in the case of $\CSAI_{m,T}$.
\end{itemize}

\begin{lemma}\label{ec}
  \begin{enumerate}
    \item 
        Let $A,B$ be models of $\CSA_{m,T}$ and let $\iota: A \rightarrow
        B$ be an embedding in the language $L_T$. If $\iota|_{Z(A)}: Z(A)
        \rightarrow Z(B)$ is existentially closed in the language $L_T$,
        then the embedding $\iota$ is existentially closed in the language
        $L_T$.
      \item 
        Let $(A,\s)$ and $(B,\tau)$ be models of $\CSAI_{m,T}$, and let
        $\iota: (A,\s) \rightarrow (B, \tau)$ be an embedding in the
        language $L_\CSAI \cup L_T$. Let $F$ and $L$ be the base fields of
        $(A,\s)$ and $(B,\tau)$, respectively. Then
      \begin{enumerate}
        \item The involutions $\s$ and $\tau$ are of the same kind;
        \item $\iota(F) \subseteq L$;
        \item If $\iota|_{F}: F \rightarrow L$
          is existentially closed in the language $L_T$, then the embedding
          $\iota$ is existentially closed in the language $L_\CSAI \cup
          L_T$.
      \end{enumerate}
  \end{enumerate}
\end{lemma}
\begin{proof}
  \begin{enumerate}
    \item Note that $\iota(Z(A)) \subseteq Z(B)$ by Lemma~\ref{csa-centre}. It
      follows that the map $A \ox_{Z(A)} Z(B) \rightarrow B$, $a \ox z
      \mapsto \iota(a)z$ is well-defined. It is also injective (since $A \ox_{Z(A)}
      Z(B)$ is simple), and thus surjective since $\dim_{Z(B)} A \ox_{Z(A)} Z(B)
      = \dim_{Z(A)} A = \dim_{Z(B)} B$. So we can assume for simplicity of
      notation that $B = A \ox_{Z(A)} Z(B)$.

      By Proposition~\ref{ec-ultra}, there is an ultrafilter $\CU$ on a set $I$
      and an $L_T$-embedding $\eta : Z(B)
      \rightarrow Z(A)^I/\CU$ such that $\eta \circ \iota = \delta$ (the diagonal
      inclusion of $Z(A)$ in $Z(A)^I/\CU$).
       The map $\eta$ extends to $\id \ox \eta$: 
      \[B = A \ox_{Z(A)} Z(B) \rightarrow A \ox_{Z(A)} (Z(A)^I/\CU), \quad a
        \ox z \mapsto a \ox \eta(z),\]
        and we define a map $\xi_1$:
      \[\xymatrix@R=0pt{
        A \ox_{Z(A)} (Z(A)^I/\CU) \ar[r] & (A \ox_{Z(A)} Z(A))^I/\CU \ar[r] & 
          A^I/\CU \\
        a \ox (z_i)_{i \in I}/\CU \ar@{|-_{>}}[r] & (a \ox z_i)_{i \in I} / \CU
          \ar@{|-_{>}}[r] & (az_i)_{i \in I} / \CU.
      }\]
      We thus have, for $\xi := \xi_1 \circ (\id \ox \eta)$:
      \[\xymatrix{
        A \ar[r]^-\iota & B = A \ox_{Z(A)} Z(B) \ar[r]^-\xi & A^I/\CU.
      }\]
      The map $\xi \circ \iota$ is the diagonal embedding of $A$ in
      $A^I/\CU$. Therefore, by Proposition~\ref{ec-ultra}, the result will
      be proved if $\xi$ is an $L_T$-embedding, which we now
      check:

      It is a morphism of rings, and therefore an embedding of rings since
      $B$ is simple. The rest of the argument relies on the
      following Fact.\medskip

      {\bf Fact~1:}
        (i) $\xi|_{Z(B)} = \eta$;

        (ii) If $b \in B \setminus Z(B)$, then $\xi(b)
          \not \in Z(A^I/\CU)$.\medskip

      Proof of Fact~1: Item (i) is a direct verification (writing $\eta(z)
      = (z_i)/\CU$ for $z \in Z(B)$). We
      consider item (ii). Let $\{1, b_2, \ldots, b_m\}$ be a $Z(B)$-basis of
      $B$, and write $b = u_1 + u_2b_2 + \cdots + u_mb_m$ with
      $u_1, \ldots, u_m \in Z(B)$, with (for instance) $u_2 \not = 0$ since
      $b \not \in Z(B)$. 

      Then
      $\xi(b) = \xi(u_1) + \xi(u_2) \xi(b_2) + \cdots + \xi(u_m)
      \xi(b_m)$,
      where 
      \begin{itemize}
        \item $\xi(u_1), \ldots, \xi(u_m) \in Z(A^I/\CU)$ and $\{1,
          \xi(b_2), \ldots, \xi(b_m)\}$ is a basis of
          $A^I/\CU$ over $Z(A^I/\CU)$ by Lemma~\ref{csa-centre}.
        \item $\xi(u_2) = \eta(u_2) \not = 0$ since $\eta$ is injective.
      \end{itemize}
      Therefore $\xi(b) \not \in Z(A^I/\CU)$.
      End of the proof of Fact~1.\medskip

      It follows from item (i) of Fact~1 and the definition of
      $\eta$ that $\xi$ is an embedding for the constant
      symbols in $L_T \setminus L_R$.

      Let $S$ be a relation symbol in $L_T \setminus L_R$ of arity $k$
      and let $b_1, \ldots, b_k \in B$. If $b_1, \ldots, b_k \in Z(B)$, we
      have $B \models S(b_1, \ldots, b_k)$ if and only if $A^I/\CU \models
      S(\xi(b_1), \ldots, \xi(b_k))$ by item (i) of Fact~1 since $\eta$ is an $L_T$-embedding.
      Otherwise, for instance if
      $b_1 \not \in Z(B)$, then $\xi(b_1) \not \in
      Z(A^I/\CU)$ by item (ii) of Fact~1 and both $S(b_1, \ldots, b_n)$ and
      $S(\xi(b_1), \ldots, \xi(b_k))$ do not hold (see the paragraph before
      the current Lemma).

      Consider now a function symbol $f \in L_T \setminus L_R$ of arity $k$
      and let $b_1, \ldots, b_k \in B$. We show that $\xi(f(b_1, \ldots,
      b_k)) = f(\xi(b_1), \ldots, \xi(b_k))$. If $b_1, \ldots, b_k \in
      Z(B)$, then $\xi(f(b_1, \ldots, b_k)) = f(\xi(b_1), \ldots,
      \xi(b_k))$ by item (i) of Fact~1 and the definition of $\eta$.
      Otherwise, for instance if $b_1 \not \in Z(B)$. Then $f(b_1, \ldots,
      b_k) = 0$ (see the paragraph before the current Lemma) and thus
      $\xi(f(b_1, \ldots, b_k)) = 0$. But $\xi(b_1) \not \in
      Z(A^I/\CU)$ by item (ii) of Fact~1, so that $f(\xi(b_1), \ldots,
      \xi(b_k)) = 0$, proving the result.

    \item (a) Assume that $\s$ is of the first kind and $\tau$ is of the second
      kind. Then $Z(A) = F$ and $Z(B)$ is a quadratic extension of $L$. It
      follows that $\dim_F A$ is a square and $\dim_L B$ is two times a square,
      contradicting $\dim_F A = \dim_L B$. The same argument shows that we
      cannot have $\s$ of the second kind and $\tau$ of the first kind.
      \medskip

      \noindent (b) If $\s$ and $\tau$ are of the first kind, then $F = Z(A)$
      and $L = Z(B)$, so that $\dim_{Z(A)} A = \dim_{Z(B)} B$. If $\s$ and
      $\tau$ are of the second kind, then $\dim_{Z(A)} A = (1/2) \dim_F A =
      (1/2) \dim_L B = \dim_{Z(B)} B$. In both cases, we obtain by
      Lemma~\ref{csa-centre} that $\iota(Z(A)) \subseteq Z(B)$. It follows that
      $\iota(F) = \iota(Z(A) \cap \Sym(A,\s)) \subseteq Z(B) \cap \Sym(B,\tau) =
      L$.\medskip

      \noindent (c) 
      The proof is along the same lines as the one for part (1), but we
      still present it in full, since some extra verifications are required
      when the involutions are of the second kind.\medskip

      {\bf Claim:} The map $g: (A \ox_F L, \s \ox \id) \rightarrow (B, \tau)$,
      $a \ox \ell \mapsto \iota(a)\ell$ is an isomorphism of $L$-algebras with
      involution.

      Proof of the Claim: We first prove that is an isomorphism of $L$-algebras,
      and we need to distinguish two cases according to the kind
      of the involutions.

      $\bullet$ If $\s$ and $\tau$ are of the first kind. Then $F = Z(A)$ and $L =
      Z(B)$, and we conclude as in the proof of the first statement of the
      Lemma.

      $\bullet$ If $\s$ and $\tau$ are of the second kind. Then $Z(A) = F(\sqrt{d})$,
      for some $d \in F$ that is not a square in $F$. Since $F$ is existentially
      closed in $L$, $d$ is not a square in $L$, and $L(\sqrt{d})$ is a field
      extension of $Z(A)$. Therefore
      \begin{align*}
        A \ox_F L &\cong A \ox_{Z(A)} Z(A) \ox_F L \cong A \ox_{Z(A)}
        (L[X]/(X^2-d)) \\
          & \cong A \ox_{Z(A)} L(\sqrt{d})
      \end{align*}
      is simple, and $g$ is an isomorphism of $L$-algebras (it is injective
      since $A \ox_F L$ is simple, and it is surjective since $\dim_L A \ox_F L
      = \dim_F A = \dim_L B$). 

      We now check that $g$ respects the involution:
      \begin{align*}
        g(\s \ox \id(a \ox \ell)) &= g(\s(a) \ox \ell) = \iota(\s(a))\ell =
        \tau(\iota(a))\ell \\
        &= \tau(\iota(a)\ell) = \tau(g(a \ox \ell)).
      \end{align*}
      End of the proof of the Claim.\medskip

      As a consequence of the Claim, and to keep notation simple, we assume
      that $(B, \tau) = (A \ox_F L, \s \ox \id)$.

      By Proposition~\ref{ec-ultra}, there is an ultrafilter $\CU$ on a set $I$
      and an $L_T$-embedding $\eta : L \rightarrow F^I/\CU$ such that $\eta
      \circ \iota = \delta$. The map $\eta$ extends to $\id \ox \eta$:
      \[B = A \ox_{F} L \rightarrow A \ox_{F} (F^I/\CU), \quad a
        \ox z \mapsto a \ox \eta(z),\]
      and we define a map $\xi_1$:
      \[\xymatrix@R=0pt{
        A \ox_{F} (F^I/\CU) \ar[r] & (A \ox_{F} F)^I/\CU \ar[r] & 
          A^I/\CU \\
        a \ox (z_i)_{i \in I}/\CU \ar@{|-_{>}}[r] & (a \ox z_i)_{i \in I} / \CU
          \ar@{|-_{>}}[r] & (az_i)_{i \in I} / \CU.
      }\]
      We thus have, for $\xi := \xi_1 \circ (\id \ox \eta)$:
      \[\xymatrix{
        A \ar[r]^-\iota & B = A \ox_{F} L \ar[r]^-\xi & A^I/\CU.
      }\]
      The map $\xi \circ \iota$ is the diagonal embedding of $A$ in
      $A^I/\CU$ therefore, by Proposition~\ref{ec-ultra}, the result will
      be proved if $\xi$ is an $L_\CSAI \cup L_T$-embedding, which we now
      check:

      The map $\xi$ is a morphism of rings, and therefore an embedding of rings since
      $B$ is simple. We first check that 
      it is a $\{\s\}$-morphism. We do this by working with $\{e_1, \ldots,
      e_t\}$, an
      $F$-basis of $A$.
      Then $\{e_1 \ox 1, \ldots, e_t \ox 1\}$ is an $L$-basis of $B$. We need to show
      that $\xi(\tau(e_j \ox 1)) = (\s^I/\CU)(\xi(e_j \ox 1))$ for $j \in \{1, \ldots,
      t\}$.

      Write $\s(e_j) = \sum_{k=1}^t u_ke_k$ for some $u_1, \ldots, u_t \in F$.
      Since, for $j \in \{1, \ldots, t\}$:
      \[\xi(e_j \ox 1) = \xi_1 \circ (\id \ox \eta)(e_j \ox 1) = \xi_1(e_j
      \ox 1) = (e_j)_{i \in I}/\CU,\]
      we obtain
      \[(\s^I/\CU)(\xi(e_j \ox 1)) = (\s(e_j))_{i \in I}/\CU = \Big(\sum_{k=1}^t
      u_ke_k\Big)_{i \in I}/\CU.\]
      We also have $\tau(e_j \ox 1) = (\s \ox \id)(e_j \ox 1) = \sum_{k=1}^t u_k
      e_k \ox 1$ and thus
      \begin{align*}
        \xi(\tau(e_j \ox 1)) &= \xi\Big(\sum_{k=1}^t u_k (e_k \ox 1)\Big) =
          \sum_{k=1}^t u_k \xi(e_k \ox 1) \\
          &= \sum_{k=1}^t u_k(e_k)_{i \in I}/\CU = \Big(\sum_{k=1}^t u_ke_k\Big)_{i \in
            I}/\CU,
      \end{align*}
      proving the desired equality.\medskip

      The rest of the argument follows the same idea as what was done in
      the first part of the Lemma, but we present the argument again, since
      there is a difference in the case where the involutions are of the
      second kind. It relies on the following Fact.\medskip

      {\bf Fact~2:}
        (i) $\xi|_{L} = \eta$;

        (ii) If $b \in B \setminus L$, then $\xi(b)
          \not \in F^I/\CU$.\medskip

      Proof of Fact 2: Item (i) is a direct verification. We
      consider item (ii). If $\s$ and $\tau$ are of the first kind, the same
      argument as in Fact~1 above applies since $F = Z(A)$ and $L = Z(B)$. So
      we assume that $\s$ and $\tau$ are of the second kind. Write $Z(B) =
      L(z)$ with $z$ such that $\tau(z) = -z$. Let $\{1, b_2,
      \ldots, b_\ell\}$ be a $Z(B)$-basis of
      $B$. Then $\{1, z, b_2, b_2z, \ldots, b_\ell, b_\ell z\}$ is an
      $L$-basis of $B$. By Lemma~\ref{csa-centre}, $\{1, \xi(b_2), \ldots,
      \xi(b_\ell)\}$ is a $Z(A^I/\CU)$-basis of $A^I/\CU$, and $\xi(z) \in
      Z(A^I/\CU)$. Observe that $Z(A^I/\CU)$ is a field extension of degree
      two of $F^I/\CU$, and that $\xi(z) \not \in F^I/\CU$ (indeed: $\xi(z)
      \not = 0$ and $\tau(z)
      = -z$, so $(\s^I/\CU)(\xi(z)) = \xi(\tau(z)) = -\xi(z) \not = \xi(z)$). Therefore
      $Z(A^I/\CU) = (F^I/\CU)(\xi(z))$, and it follows that the set $\{1, \xi(z),
      \xi(b_2), \xi(b_2)\xi(z), \ldots, \xi(b_\ell), \xi(b_\ell)\xi(z)\}$
      is an $F^I/\CU$-basis of $A^I/\CU$.

      Renaming the elements, we have an $L$-basis $\{1,c_2, \ldots, c_t\}$ of $B$ such
      that $\{1, \xi(c_2), \ldots, \xi(c_t)\}$ is an $F^I/\CU$-basis of
      $A^I/\CU$. Since $b \not \in L$, we have 
      $b = u_1 + u_2c_2 + \cdots + u_t c_t$ with
      $u_1, \ldots, u_t \in L$, and (for instance) $u_2 \not = 0$. Then $\xi(u_2) \not = 0$ and
      $\xi(b) = \xi(u_1) + \xi(u_2)
      \xi(c_2) + \cdots + \xi(u_t) \xi(c_t)$, which proves the
      result, since $\{1, \xi(c_2), \ldots, \xi(c_t)\}$ is an $F^I/\CU$-basis of
      $A^I/\CU$. End of the proof of Fact~2.\medskip

      It follows from item (i) of Fact~2 and the definition of
      $\eta$ that $\xi$ is an embedding for the constant and relation
      symbols in $L_T \setminus L_R$.
      
      Let $S$ be a relation symbol in $L_T \setminus L_R$ of arity $k$
      and let $b_1, \ldots, b_k \in B$. If $b_1, \ldots, b_k \in L$, we
      have $B \models S(b_1, \ldots, b_k)$ if and only if $A^I/\CU \models
      S(\xi(b_1), \ldots, \xi(b_k))$ by item (i) of Fact~2 since $\eta$ is an $L_T$-embedding.
      Otherwise, for instance if
      $b_1 \not \in L$, then $\xi(b_1) \not \in
      F^I/\CU$ by item (ii) of Fact~2 and both $S(b_1, \ldots, b_n)$ and
      $S(\xi(b_1), \ldots, \xi(b_k))$ do not hold (see the paragraph before
      the current Lemma).

      Consider now a function symbol $f \in L_T \setminus L_R$ of arity $k$
      and let $b_1, \ldots, b_k \in B$. We show that $\xi(f(b_1, \ldots,
      b_k)) = f(\xi(b_1), \ldots, \xi(b_k))$. If $b_1, \ldots, b_k \in L$,
      then $\xi(f(b_1, \ldots, b_k)) = f(\xi(b_1), \ldots, \xi(b_k))$ by
      item (i) of Fact~2 and the definition of $\eta$.  Otherwise, for
      instance if $b_1 \not \in L$. Then $f(b_1, \ldots, b_k) = 0$ (see the
      paragraph before the current Lemma) and thus $\xi(f(b_1, \ldots, b_k)) =
      0$. Furthermore, $\xi(b_1) \not \in F^I/\CU$ by item (ii) of
      Fact~2, so that $f(\xi(b_1), \ldots, \xi(b_k)) = 0$, proving the
      result. \qedhere
  \end{enumerate}
\end{proof}

\begin{thm}\label{mod-comp-ec}
  \begin{enumerate}
    \item Assume that $T$ is model-complete. Then the theories $\CSA_{m,T}$ and
      $\CSAI_{m,T}$ are model-complete in the languages $L_T$ and
      $L_\CSAI \cup L_T$, respectively.
    \item Conversely, if one of the theories $\CSA_{m,T}$ or $\CSAI_{m,T}$ is
      model-complete, then $T$ is model-complete.
  \end{enumerate}
\end{thm}
\begin{proof}
  \begin{enumerate}
    \item We first consider the case of the theory $\CSA_{m,T}$. By
      Robinson's test (cf.  \cite[Theorem~3.5.1]{chang-keisler}), it
      suffices to show that if $\CM$ and $\CN$ are models of $\CSA_{m,T}$
      and $\CM$ is an $L_T$-substructure of $\CN$, then $\CM$ is
      existentially closed in $\CN$. In such a case we have $Z(\CM)
      \subseteq Z(\CN)$ by Lemma~\ref{csa-centre}, and it follows that
      $Z(\CM)$ is an $L_T$-substructure of $Z(\CN)$ (using that $\CM$ is an
      $L_T$-substructure of $\CN$). Since both $Z(\CM)$ and $Z(\CN)$ are
      models of $T$ by hypothesis, and $T$ is model-complete, $Z(\CM)$ is
      existentially closed in $Z(\CN)$ in the language $L_T$, and the
      result follows by Lemma~\ref{ec}(1).
    
      We now consider the theory $\CSAI_{m,T}$: Let $\CM$, $\CN$ be two
      models of $\CSAI_{m,T}$ with $\CM$ an $L_\CSAI \cup L_T$-substructure
      of $\CN$. Let $F$ and $L$ be the base fields of $\CM$ and $\CN$,
      respectively.
      By Lemma~\ref{ec}(2) we have $F \subseteq L$, and it follows that $F$
      is an $L_T$-substructure of $L$ (using that $\CM$ is an
      $L_T$-substructure of $\CN$). This inclusion is existentially closed
      in $L_T$ since 
      $F$ and $L$ are models of $T$. Using
      Lemma~\ref{ec}(2) again, the inclusion of $\CM$ in $\CN$ is
      existentially closed in $L_\CSAI \cup L_T$, and the result follows by
      Robinson's test.
    \item Let $F \subseteq L$ be models of $T$. If $\CSA_{m,T}$ is
      model-complete: $M_n(F) \subseteq M_n(L)$ are models of $\CSA_{m,T}$ (where
      $m=n^2$), so $M_n(F) \prec M_n(L)$, from which follows $F \prec L$. If
      $\CSAI_{m,T}$ is model-complete, the same argument applies, but with
      $(M_n(F), {}^t) \subseteq (M_n(L), {}^t)$. \qedhere
  \end{enumerate}
\end{proof}

\begin{cor}
  Assume that $T_0$ and $T_1$ are theories of fields, with $T_1$ the
  model-companion of $T_0$. Then the theory $\CSA_{m,T_1}$ is the
  model-companion of $\CSA_{m,T_0}$.
\end{cor}
\begin{proof}
  It follows from Theorem~\ref{mod-comp-ec}, since $\CSA_{m,T_0}$ and
  $\CSA_{m,T_1}$ are clearly cotheories: Let $i \not = j \in \{0,1\}$. If $A$ is
  a model of $\CSA_{m,T_i}$, take for $K$ a model of $T_j$ containing $Z(A)$,
  then $A \subseteq A \ox_{Z(A)} K$, and $A \ox_{Z(A)} K$ is a model of
  $\CSA_{m,T_j}$.
\end{proof}

\begin{rem}\label{rem1}
  Let $\of$ be the theory of ordered fields. It is well-known that $\rcf$ is the
  model-companion of $\of$.

  However, the theory $\CSAI_{m,\rcf}$ is not the model-companion of $\CSAI_{m,\of}$
  since not every model of $\CSAI_{m,\of}$ can be embedded in a model of
  $\CSAI_{m,\rcf}$ in the language $L_\CSAI \cup \{\lP\}$. The problem arises
  when the centre is a proper ordered extension of the base field.

  Let $\CA:=(A,\s) \models \CSAI_{m,\of}$ with $F \subsetneq Z(A) =
  F(\alpha_0)$, where $F$ is the base field of $(A,\s)$.
  We can assume that $\alpha_0$ is a root of
  a polynomial $X^2-d \in F[X]$. The problem occurs when $d \in P := \lP^\CA$
  (the ordering on $F$).  Since $\alpha_0 \not \in F$, we have $\s(\alpha_0)
  \not = \alpha_0$, and thus $\s(\alpha_0) = -\alpha_0$ since $\alpha_0$ is a root of
  $X^2-d$. It follows that $(A,\s)$ cannot be an $L_\CSAI \cup
  \{\lP\}$-substructure of a model $\CB := (B,\tau)$ of $\CSAI_{m,\rcf}$:
  Let $(B,\tau)$ be an $L_\CSAI \cup \{\lP\}$-structure such that $(A,\s)$ is an
  $L_\CSAI \cup \{\lP\}$-substructure of $(B,\tau)$ and $\lP^\CB$ is an ordering
  on 
  the base field of $(B,\tau)$, which we denote by $L$.
  We show that $X^2-d$ has no root in 
  $L$, so that
  $(L,\lP^\CB)$ is not a real closed extension of $(F,P)$. Let $\alpha_1$ be a
  root of $X^2-d$ in $Z(B)$ (if it exists). Then $\alpha_0, -\alpha_0, \alpha_1$
  are roots of $X^2-d$ in the field $Z(B)(\alpha_0)$. Therefore $\alpha_1 = \pm
  \alpha_0$ is skew-symmetric (i.e., $\tau(\alpha_1) = -\alpha_1$), and in
  particular does not belong to $L$.
\end{rem}

\begin{rem}\label{star}
  This situation ($Z(A) = F(\alpha_0)$ where $\alpha_0$ is a root of $X^2-d$
  with $d \in P$) leads to another problem when extending the scalars in $A$
  from $F$ to a real closure $L$ of $(F,P)$. Since $d$ has a square root in $L$,
  we obtain
  \begin{align*}
    A \ox_F L &\cong A \ox_{Z(A)} Z(A) \ox_F L \cong A \ox_{Z(A)} L[X]/(X^2-d) \\
      &\cong (A \ox_{Z(A)} L) \x (A \ox_{Z(A)} L),
  \end{align*}
  which is not a simple algebra anymore. This situation does not occur when
  $(A,\s)$ is equipped with a positive cone over $P$, cf.
  Lemma~\ref{ext-simple}.
\end{rem}

As pointed out in Remark~\ref{star}, the problem occurring in
Remark~\ref{rem1} comes from the fact that extending the base field can
turn the algebra into a product. This is actually the only obstacle to an
embedding result for models of $\CSAI_{m,T}$:
\begin{prop}
  Let $T_1$ and $T_0$ be theories of fields in the language $L_T$, with $T_1$ model-companion of
  $T_0$. Let $(A,\s)$ be a model of $\CSAI_{m,T_0}$ with base field $F$.
  Then $(A,\s)$ embeds into a model of $\CSAI_{m,T_1}$ if and only if $F$
  embeds in a model $L$ of $T_1$ such that $Z(A) \ox_F L$ is a field.
\end{prop}
\begin{proof}

  ``$\Rightarrow$'' Let $\iota : (A,\s) \rightarrow (B,\tau)$ be an
  $L_{\CSAI} \cup L_T$-embedding, where $(B,\tau)$ is a model of
  $\CSAI_{m,T_1}$ with base field $L$. The Claim from the proof of
  Lemma~\ref{ec}(2)(c) gives $A \ox_F L \cong B$, and thus $Z(A \ox_F L)
  \cong Z(B)$. Applying \cite[Section~7, Corollary~3]{Draxl83}, we get
  $Z(A) \ox_F L \cong Z(B)$, which is a field since $B$ is simple.

  ``$\Leftarrow$'' The canonical map $(A,\s) \rightarrow (A \ox_F L, \s \ox
  \id)$ is an $L_\CSAI \cup L_T$-embedding (see the definition of $(A,\s)$
  as $L_\CSAI \cup L_T$-structure just before Lemma~\ref{ec}). Furthermore,
  the
  algebra $A \ox_F L$ is central simple with
  centre $Z(A) \ox_F L$ since $Z(A) \ox_F L$ is a field (by
  \cite[Section~7, Corollary~6]{Draxl83} since $A \ox_F L \cong (A \ox_{Z(A)} Z(A))
  \ox_F L \cong A \ox_{Z(A)} (Z(A) \ox_F L)$). It follows that $(A \ox_F L, \s
  \ox \id)$ is a model of $\CSAI_{m, T_1}$.
\end{proof}

The special case where $T$ is the theory $\rcf$ of real closed fields in
the language $L_R$ will be important in the remainder of this paper, and we can
get a more precise statement than Theorem~\ref{mod-comp-ec} in this case:

\begin{lemma}\label{inclusion-invo-ec}
  Let $(D, \vt) \in \{(F, \id), (F(\sqrt{-1}), -), ((-1,-1)_F, -)\}$ with $F$
  real closed. Let $(A,\s)$ be an $F$-algebra with involution and $f : A
  \rightarrow M_n(D)$ be an isomorphism of $F$-algebras. Let $(B, \tau)$ be an
  $L$-algebra with involution for some real closed field $L$, such that
  $(A,\s)$ is an $L_\CSAI$-substructure of $(B,\tau)$ (i.e., 
  $(A,\s) \subseteq (B, \tau)$), $F \subseteq L$, and $\dim_F A = \dim_L B$. Then:
  \begin{enumerate}
    \item\label{inclusion-invo2} There are an isomorphism of $L$-algebras $g$ and $a \in M_n(D)^\x$
      such that, if
      \[(E,\vt') := \begin{cases} (L, \id) & \text{if } (D, \vt) = (F, \id) \\
        (L(\sqrt{-1}), -) & \text{if } (D, \vt) = (F(\sqrt{-1}),-) \\
        ((-1,-1)_L, -) & \text{if } (D, \vt) = ((-1,-1)_F, -)
      \end{cases}\]
      then the following diagram is commutative
      \[\xymatrix{
        (A,\s) \ar[r]^-f \ar[d]_{\subseteq} & (M_n(D), \Int(a) \circ \vt^t) \ar[d]^{\subseteq} \\
        (B, \tau) \ar[r]_-g & (M_n(E), \Int(a) \circ (\vt')^t) 
      }\]
      where the inclusion on the
      right is the canonical one induced by $F \subseteq L$, and all maps
      respect the reduced trace $\Trd$.
    \item\label{inclusion-invo3} The inclusion of $(A, \s)$ in $(B,\tau)$ is
      elementary in the language $L_\CSAI \cup \{\lP,\lTrd\}$ (where $\lP$ and
      $\lTrd$ are interpreted in each structure by the unique ordering on 
      the base field
      and by the reduced trace, respectively).
  \end{enumerate}
\end{lemma}
\begin{proof}
  \begin{enumerate}
    \item Define $(E, \vt') := (D \ox_F L, \vt \ox \id)$. Note that $E$ is as
      described in the statement of the proposition. Since $L$ is real closed
      and $(E, \vt')$ is definable as $L_\CSAI$-structure in the same way in $L$
      as $(D, \vt)$ is in $F$, we get that the inclusion of $(D, \vt)$ in $(E,
      \vt')$ is elementary in $L_\CSAI$. \medskip
    
      {\bf Step 1}: There is a diagram as claimed, without the involutions.

      The map $h : A \ox_F L \rightarrow B$, $a \ox \ell \mapsto al$ is well-defined and
      an isomorphism of $L$-algebras (it is injective since $A$, being an
      $F$-vector space, is flat over $F$, and
      surjective since $\dim_L A \ox_F L = \dim_F A = \dim_L B$). Tensoring by $L$
      over $F$, the isomorphism $f :
      A \rightarrow M_n(D)$ yields an isomorphism of $L$-algebras $f \ox \id : A
      \ox_F L \rightarrow M_n(D) \ox_F L$, so that we have a diagram
      \[\xymatrix{
        A \ar[d]_-{\subseteq} \ar[rrr]^-{f} & & & M_n(D) \ar[d]^{\lambda} \\
        B \ar[r]^-{h^{-1}} & 
        A \ox_F L \ar[r]^-{f \ox \id} & M_n(D) \ox_F L \ar[r]^-{\mu} & M_n(E)
      }\]
      where $\lambda$ is the map induced by the inclusion of $D$ in $E$, and
      $\mu$ is the canonical identification of $M_n(D) \ox_F L$ and $M_n(D \ox_F
      L)$. 

      We immediately see that the diagram is commutative, showing the
      result for $g := \mu \circ (f \ox \id) \circ h^{-1}$. \medskip

      {\bf Step 2}: We add the involutions to the diagram.

      Let $\s'$ be the involution on $M_n(D)$ such that $f : (A,\s) \rightarrow
      (M_n(D), \s')$ is an isomorphism of algebras with involution (i.e.,
      $\s' \circ f = f \circ \s$), and $\tau'$ be the involution on $M_n(E)$ such
      that $g : (B, \tau) \rightarrow (M_n(E), \tau')$ is an isomorphism of
      algebras with involution (i.e., $\tau' \circ g = g \circ \tau$). Note
      that $\s'$ is of the same kind as $\s$
      and that $\tau'$ is of the same kind as $\tau$.
      Using that $g$
      extends $f$, we obtain that $\tau'$ extends $\s'$, so that $(M_n(D), \s')
      \subseteq (M_n(E), \tau')$. 
  
      By the Skolem-Noether theorem (cf. Proposition~\ref{SK-invo}), $\s' =
      \Int(a) \circ \vt^t$ for some $a \in M_n(D)^\x$, and $\tau' = \Int(b)
      \circ {\vt'}^t$ for some $b \in M_n(E)^\x$. Since $\tau'$ extends
      $\s'$ we have $\Int(a) \circ \vt^t = \Int(b) \circ (\vt')^t$ on
      $M_n(D)$, so that $\Int(a) = \Int(b)$ on $M_n(D)$.  If $\CB$ is any
      basis of $M_n(D)$ over $F$, then $\CB$ is a basis of $M_n(E)$ over
      $L$, and $\Int(a) = \Int(b)$ on $\CB$, so that $\Int(a) = \Int(b)$ on
      $M_n(E)$.  In particular we can take $a=b$, and the diagram indicated
      in the statement is commutative in the language $L_\CSAI$.
  
      We still need to check that the maps respect the reduced trace: It is the
      case for $f$ and $g$ by Lemma~\ref{trace-morphism}, it is clear for the
      canonical inclusion of $M_n(D)$ in $M_n(E)$, and it is therefore also the
      case for the inclusion of $(A, \s)$ into $(B, \tau)$ since the diagram
      commutes.
  
    \item Since the two $L_\CSAI \cup \{\lP, \lTrd\}$-structures $(M_n(D),
      \Int(a) \circ \vt^t)$ and $(M_n(E), \Int(a) \circ (\vt')^t)$ are
      interpretable in the same way in $F$ and in $L$, and $F \prec L$, we have
      $(M_n(D), \Int(a) \circ \vt^t) \prec (M_n(E), \Int(a) \circ (\vt')^t)$ in
      $L_\CSAI$. The result follows because of the commutativity of the diagram.
      \qedhere
   \end{enumerate}
\end{proof}

\section{Quantifier elimination}

We turn our attention to quantifier elimination with the help of a simple
extension of \cite[Theorem~2.2.4]{kt25} (our Proposition~\ref{qe-mat}). We then
point out some potential problems, which can be avoided by introducing
``orderings'' (positive cones) on central simple algebras with involution, and specifying the
type of the involution (Theorem~\ref{KT-eq}).

\begin{defi}
  Let $(D, \vt)$ be a division algebra with involution over a formally real
  field $F$. We say that $(D, \vt)$ is
  \begin{enumerate}
    \item of real type if $(D, \vt) = (F, \id)$;
    \item of complex type if $(D, \vt) = (F(\sqrt{-1}), -)$;
    \item of quaternion type if $(D, \vt) = ((-1,-1)_F, -)$.
  \end{enumerate}
\end{defi}

\begin{lemma}\label{ee-isom}
  Let $(D,\vt) \in \{(F,\id), (F(\sqrt{-1}),-), ((-1,-1)_F,-)\}$ where $F$ is a
  real closed field. Let $\CM$ be elementarily equivalent to $(M_n(D),\vt^t)$ in
  $L := L_\CSAI \cup \{\lP,\lTrd\}$, where $\lP$ is interpreted in $M_n(D)$ by
  the ordering on $F$, and $\lTrd$ by the reduced trace map. Let $F^\CM$ be the
  base field of $\CM$ (note that it is real closed).

  Then there is an $L$-isomorphism $\phi_\CM : \CM \rightarrow (M_n(D_\CM),
  (\vt_\CM)^t, \Trd)$, where $(D_\CM, \vt_\CM)$ is a sub-division algebra with
  involution of $\CM$ over $F^\CM$ that is of the same type as $(D, \vt)$.
\end{lemma}
\begin{proof}
  Let $\CN$ be the $L_\CSAI \cup \{\lP, \lTrd\}$-structure $(M_n(D),\vt^t)$.  It
  follows from Lemma~\ref{axiom-Trd} that ``$\lTrd \text{ is the reduced
  trace}$'' can be expressed by a first-order formula (depending on $n$ and $D$).

  Let $\CB := \{E_{rs}\}_{r,s=1}^n$ be the canonical basis of $M_n(F)$ (the
  matrix $E_{rs}$ is the matrix with zeroes everywhere, except for a 1 at
  coordinates $(r,s)$). We have
  \[M_n(F) \models \delta(E_{rs})_{r,s \in \{1,\ldots,n\}},\] 
  where
  \begin{align*}
    \delta(X_{r,s})_{r,s \in \{1,\ldots,n\}} := \bigwedge_{r,s,t} X_{r,s} \cdot X_{s,t} =
    X_{r,t} \not = 0 \wedge \bigwedge_{r,s,t,\ell, s \not = t} X_{r,s} \cdot
    X_{t,\ell} = 0,
  \end{align*}
  so that $\delta$ describes the multiplicative properties of the elements of
  $\CB$.

  We define three $L_\CSAI \cup \{\lP, \lTrd\}$-formulas; in each of them,
  $\lF$
  denotes the base field (which is clearly definable in $L_\CSAI$):
  \begin{align*}
    \Omega_{\text{orthogonal}} :=&\ (
    \text{centre = } \lF) \wedge \lTrd \text{
      is the reduced trace} \wedge \\
    &\exists e_{11}, \ldots, e_{nn} \
    \Bigl\{\delta(e_{11}, \ldots, e_{nn}) \wedge \{e_{11}, \ldots, e_{nn}\}
    \text{ is an $\lF$-basis } \wedge \\
    &\qquad \qquad \qquad \ls = {}^t \Bigr\}\\
    \Omega_{\text{unitary}} :=&\ \exists i \ \Bigl\{i^2=-1 \wedge (\text{centre
      = } \lF(i)) \wedge \lTrd \text{ is the reduced trace } \wedge \\
    &\exists e_{11}, \ldots, e_{nn} \
      \delta(e_{11}, \ldots, e_{nn}) \wedge \\
    &\{e_{11}, \ldots, e_{nn}\} \text{ is a basis over the centre } \wedge 
      \ls = -^t \Bigr\} \\
    \Omega_{\text{symplectic}} :=&\ (\text{centre = } \lF) \wedge \lTrd \text{
      is the reduced trace } \wedge \\
    &\exists i,j,k \exists e_{11}, \ldots, e_{nn}\ \Bigl\{ i^2=j^2=-1 \wedge
    ij=-ji=k \ \wedge \\
    & \Span_{\lF}\{1,i,j,k\} \text{ is a division algebra}\ \wedge
     \delta(e_{11}, \ldots, e_{nn})\  \wedge\\
    & i,j,k \text{ commute with }e_{11}, \ldots, e_{nn}\ \wedge \\
    & \{e_{11}, \ldots, e_{nn}\}
      \text{ is a basis over } \Span_{\lF}\{1,i,j,k\} \wedge 
      \ls = -^t\Bigr\}
  \end{align*}
  Let $\Omega$ be the only one of the above three formulas such that $\CN
  \models \Omega$. Then $\CM \models \Omega$ and the map $\phi_\CM$ is $\CM
  \rightarrow M_n(D_\CM)$, $d \in D_\CM \mapsto d$, $e_{rs} \mapsto E_{rs}$ (the
  matrix with 1 at entry $(r,s)$ and $0$ elsewhere). The map $\phi_\CM$ is an
  isomorphism of 
  $F^\CM$-algebras since the formula $\Omega$ specifies the
  structure constants. It therefore respects the reduced trace, and it is clear
  that it respects the ordering on $F^\CM$. Finally, $F^\CM$ is real closed
  since $F^\CM \equiv F$.
\end{proof}

The following amalgamation result is the main argument in the proof of
\cite[Theorem~2.2.4]{kt25} in the case of matrices over real closed fields and
algebraically closed fields, but also covers the additional case of matrices
over real quaternions. While the proof is the same as in
\cite[Theorem~2.2.4]{kt25} (using Theorem~\ref{Specht-property} instead of
\cite[Theorem~2.2.2]{kt25} in order to include the quaternion case), we still
reproduce it for the convenience of the reader, since it
is reasonably short.

\begin{lemma}\label{amalg-arg}
  Let $(D_i,\vt_i) \in \{(F_i,\id), (F_i(\sqrt{-1}),-),
  ((-1,-1)_{F_i},-)\}$ with $F_i$ a real closed field, for $i=1,2$, 
  and such that $(D_1, {\vt_1})$ and $(D_2,
    {\vt_2})$ are of the same type.
  Let $L := L_\CSAI \cup \{\lP, \lTrd\}$,
  where $\lP$ is interpreted
  as the ordering on $F_i$, and $\lTrd$ as the reduced trace map.
  Let $\CA$ be a finitely generated $L$-structure and $\iota_1 : \CA \rightarrow
  (M_n(D_1), {\vt_1}^t)$, $\iota_2 : \CA \rightarrow (M_n(D_2), {\vt_2}^t)$ be
  $L$-embeddings.

  Then there are $(D_3, {\vt_3}) \in \{(F_3,\id), (F_3(\sqrt{-1}),-),
  ((-1,-1)_{F_3},-)\}$ of the same type as $(D_1, {\vt_1})$, with $F_3$ a real
  closed field, and $L$-morphisms $\lambda_i : (M_n(D_i), {\vt_i}^t) \rightarrow
  (M_n(D_3), {\vt_3}^t)$ such that the following diagram commutes:
  \[\xymatrix{
        & (M_n(D_3), {\vt_3}^t) &   \\
     (M_n(D_1), {\vt_1}^t) \ar[ur]^{\lambda_1} & & (M_n(D_1), {\vt_2}^t)
       \ar[ul]_{\lambda_2} \\
        & \CA \ar[ul]^{\iota_1} \ar[ur]_{\iota_2} &
  }\]
\end{lemma}

\begin{proof}
  Let $R$ be the subring of $\CA$ generated by the image of $\lTrd^\CA$. It is
  commutative since $\iota_1(R)$ is included in the field $F_1$. Since the
  theory of real closed fields has quantifier elimination, it has the
  amalgamation property over substructures
  (\cite[Proposition~3.5.19]{chang-keisler}), so there are a real closed field
  $F_3$ and $L_R$-embeddings $\ve: F_1 \rightarrow F_3$, $\delta : F_2
  \rightarrow F_3$ such that the following diagram commutes:
  \[\xymatrix{
       & F_3 &   \\
     F_1 \ar[ur]^\ve &  & F_2 \ar[ul]_\delta \\
       & R \ar[ul]^{\iota_1|_R} \ar[ur]_{\iota_2|_R}
     }\]
   Let $(D_3, \vt_3)$ be the division algebra with involution over
   $F_3$ of the
   same type as $(D_1,\vt_1)$ (and thus as $(D_2, \vt_2)$). Both $\ve$ and $\delta$ induce canonical maps
   from $(D_1, \vt_1)$ to $(D_3, \vt_3)$ and from $(D_2, \vt_2)$ to
   $(D_3, \vt_3)$, and thus canonical maps
   \[\ve_n : (M_n(D_1), {\vt_1}^t) \rightarrow (M_n(D_3), {\vt_3}^t)
   \text{ and } \delta_n : (M_n(D_2), {\vt_2}^t)
   \rightarrow (M_n(D_3), {\vt_3}^t).\]
   It is clear that $\ve_n$ and $\delta_n$ are
   $L_\CSAI \cup \{\lP, \lTrd\}$-embeddings. We thus have\medskip
  \[\xymatrix{
    (M_n(D_3), {\vt_3}^t) & & (M_n(D_3), {\vt_3}^t) \\
    (M_n(D_1), {\vt_1}^t) \ar[u]^{\ve_n} & & (M_n(D_2), {\vt_2}^t)
    \ar[u]_{\delta_n}\\
        & \CA \ar[ul]^{\iota_1} \ar[ur]_{\iota_2} & 
      }\]

   Since $\CA$ is finitely generated, there are $x_1, \ldots, x_d \in \CA$ such
   that $\CA$ is generated by $x_1, \ldots, x_d$ as $L$-structure.\medskip

   {\bf Claim:} There is $O \in M_n(D_3)$ such that $\vt_3(O)^t O = I_n$ and
       \[\vt_3(O)^t \ve_n(\iota_1(x_i)) O = \delta_n(\iota_2(x_i))\]
       for $i=1, \ldots, d$.

   {\bf Proof of the claim:} Let $Y_i := \ve_n(\iota_1(x_i))$ and $Z_i :=
   \delta_n(\iota_2(x_i))$. Write $x := (x_1, \ldots, x_d)$, $Y := (Y_1, \ldots,
   Y_d)$ and $Z := (Z_1, \ldots, Z_d)$.  By Theorem~\ref{Specht-property} it
   suffices to show that
   \begin{equation}
       \label{to-show}
       \Trd_{M_n(D_3)}\big(w(Y, \vt_3(Y)^t)\big)
       =
        \Trd_{M_n(D_3)}\big(w(Z, \vt_3(Z)^t)\big)
   \end{equation}
   for every word $w$ in $2d$ variables.
  We have
   \begin{align*}
     \Trd_{M_n(D_3)}\big(w(Y, \vt_3(Y)^t)\big) 
       &= \Trd_{M_n(D_3)}\big(w(\ve_n \circ \iota_1(x),
       {\vt_3}^t \circ \ve_n \circ \iota_1(x)\big) \\ 
       &= \Trd_{M_n(D_3)}\big(w(\ve_n \circ \iota_1(x),
       \ve_n \circ {\vt_1}^t \circ \iota_1(x)\big) \\
       & \qquad \text{(since ${\vt_3}^t \circ \ve_n = \ve_n \circ {\vt_2}^t$)} \\
       &= \ve\big[\Trd_{M_n(D_1)}(w(\iota_1(x), {\vt_1}^t \circ \iota_1(x))\big] \\
       &\qquad \text{(by definition of $\ve_n$)} \\
       &= \ve \circ \iota_1\big[\lTrd^\CA(w(x, \ls^\CA(x))\big] \\
       &\qquad \text{(since $\iota_1$ is an $L$-morphism)} \\
       &= \delta \circ \iota_2\big[\lTrd^\CA(w(x, \ls^\CA(x))\big] 
   \end{align*}
   while
   \begin{align*}
     \Trd_{M_n(D_3)}\big(w(Z, \vt_3(Z)^t)\big) 
       &= \Trd_{M_n(D_3)}\big(w(\delta_n \circ \iota_2(x),
       {\vt_3}^t \circ \delta_n \circ \iota_2(x)\big) \\ 
       &= \Trd_{M_n(D_3)}\big(w(\delta_n \circ \iota_2(x),
       \delta_n \circ {\vt_2}^t \circ \iota_2(x)\big) \\
       & \qquad \text{(since ${\vt_3}^t \circ \delta_n = \delta_n \circ {\vt_2}^t$)} \\
       &= \delta\big[\Trd_{M_n(D_1)}(w(\iota_2(x), {\vt_2}^t \circ \iota_2(x))\big] \\
       &\qquad \text{(by definition of $\delta_n$)} \\
       &= \delta \circ \iota_2\big[\lTrd^\CA(w(x, \ls^\CA(x))\big] \\
       &\qquad \text{(since $\iota_1$ is an $L$-morphism)},
   \end{align*}
   proving \eqref{to-show}.
   {\bf End of the proof of the Claim.}\medskip

   Taking $O$ as in the claim, the map $\Int(O) : M_n(D_3) \rightarrow
   M_n(D_3)$, $M \mapsto {\vt_3(O)}^t M O$ is an $L$-automorphism of
   $M_n(D_3)$ (since the reduced trace is invariant under isomorphisms of
   $F_3$-algebras), and we have $\xi \circ \ve_n \circ \iota_1
   = \delta_n \circ \iota_2$, proving the result.
\end{proof}

\begin{prop}\label{qe-mat}
  Let $(D,\vt) \in \{(F,\id), (F(\sqrt{-1}),-), ((-1,-1)_F,-)\}$, where $F$ is a
  real closed field. Let $\Sigma$ be the theory (without parameters) of
  $(M_n(D),\vt^t)$ in $L := L_\CSAI \cup \{\lP, \lTrd\}$, where $\lP$ is
  interpreted as the ordering on $F$, and $\lTrd$ as the reduced trace map.

  Then $\Sigma$ has quantifier elimination in $L$.
\end{prop}

\begin{proof}
  We first observe that $\Sigma$ is model-complete: Let $\CM$, $\CN$ be models
  of $\Sigma$
  such that $\CM \subseteq \CN$ as $L$-structures. Then $\CM$, $\CN$ are models
  of $\CSAI_{m,\rcf}$ for some $m$ and thus $\CM \prec \CN$ in $L_\CSAI \cup
  \{\lP, \lTrd\}$ by Theorem~\ref{mod-comp-ec}.

  Therefore, it suffices to show that $\Sigma$ has the amalgamation
  property over finitely generated substructures (this follows from
  \cite[Theorem~3.1.4]{marker02}, where the proof of ii) implies i) shows that
  the common substructure can be taken finitely generated). Let $\CM, \CN
  \models \Sigma$ and let $\CA$ be a finitely generated $L$-substructure of $\CM$ and
  $\CN$. Let $F_\CM$ and $F_\CN$ be the base fields of $\CM$ and $\CN$,
  respectively.

  By Lemma~\ref{ee-isom}, there are $L$-isomorphisms
  \[\phi : \CM \rightarrow (M_n(D_\CM), {\vt_\CM}^t, \Trd) \text{ and } \psi : \CN
  \rightarrow (M_n(D_\CN), {\vt_\CN}^t, \Trd)\]
  where $(D_\CM, \vt_\CM)$ and $(D_\CN, \vt_\CN)$ are over the real closed
  fields $F_\CM$ and $F_\CN$, respectively, and are of the same type as $(D,
  \vt)$. Therefore, we have
  \[\xymatrix{
    (M_n(D_\CM), {\vt_\CM}^t, \Trd) & & (M_n(D_\CN), {\vt_\CN}^t, \Trd) \\
    \CM \ar[u]^\phi & & \CN \ar[u]_\psi \\
    & \CA \ar[ul] \ar[ur]
  }\]
  and we conclude by Lemma~\ref{amalg-arg}.
\end{proof}

We will now consider quantifier elimination for some theories of central simple
algebras with involution. We are interested in situations where the base field
is ordered, so we will naturally have these theories specify that it is real
closed. The main two difficulties are presented in the following remark.

\begin{rem}\label{qe-problems}
  Assume that $T$ is a theory of central simple algebra with involution, such
  that $\CSAI_{m,\rcf} \subseteq T$ and $T$ admits quantifier elimination:
  
  \begin{enumerate}
    \item\label{problem1} If structures like $(M_{2n}(F),{}^t)$ and
      $(M_n((-1,-1)_F), -^t)$ are models
      of $T$. In this case the amalgamation property applied to
      \[\xymatrix{
        (M_{2n}(F),{}^t) & & (M_n((-1,-1)_F), -^t) \\
        & (\Q, \id) \ar[ul] \ar[ur]
      }\]
      would give a common elementary extension of the structures $(M_{2n}(F),{}^t)$ and
      $(M_n((-1,-1)_F), -^t)$, which is not possible since the two involutions are
      of different types (which can be expressed by a first-order property in
      the language $L_R \cup \{\ls\}$). So we need to avoid such (or similar) situations.
  
    \item\label{problem2} A model $\CM$ of $T$ 
      whose base field is real closed will
      be isomorphic to $(M_n(D), \Int(a) \circ \vt^t)$ (using the notation of
      Lemma~\ref{inclusion-invo-ec}) for some $a \in M_n(D)^\x$. In order to use
      Proposition~\ref{qe-mat}, we would like to be able to go back to an 
      algebra of the type $(M_n(D'), {\vt'}^t)$ (for some division algebra with
      involution $(D', \vt')$ of the same type as $(D, \vt)$), so we need a way
      to get some control on how the involution is scaled by $\Int(a)$.
  
  \end{enumerate}
\end{rem}
We will see in Section~\ref{with-type} that specifying the type of the
involution (logically) solves the first problem. The second will be solved by
introducing a positive cone, which will (on top of specifying an ordering on the
base field) give some control over the involution, due to the links between
positive cones and positive involutions
(Remark~\ref{results-pc}\eqref{results-pc1} and
Corollary~\ref{bar-tr}\eqref{bar-tr1}).

\subsection{Positive cones and positive involutions}

Positive cones on central simple algebras with involution have been introduced in \cite{au20}
as an attempt to define a notion of ordering that corresponds to signatures of
hermitian forms. They are also closely linked to positive involutions.
\begin{defi}[{\cite[Definition~3.1]{au20}}]\label{def-preordering}
  Let $(A,\s)$ be a central simple algebra with involution over $F$. A
  \emph{prepositive cone $\CP$ on }$(A,\s)$ is a subset $\CP$ of $\Sym(A,\s)$
  such that 
  \begin{enumerate}[(P1)]
    \item $\CP \not = \varnothing$;
    \item $\CP + \CP \subseteq \CP$; 
    \item $\s(a) \cdot \CP \cdot a \subseteq \CP$ for every $a \in A$;
    \item $\CP_F := \{u \in F \mid u\CP \subseteq \CP\}$ is an ordering on $F$.
    \item $\CP \cap -\CP = \{0\}$ (we say that $\CP$ is proper).
  \end{enumerate}
  A prepositive cone $\CP$ is over $P\in X_F$ (the set of all orderings on $F$)
  if $\CP_F=P$, and a positive cone is a prepositive cone that is  maximal with
  respect to inclusion. 
\end{defi}

We recall the following long list of results about positive cones. Most of them
are direct or appear in some other papers.
\begin{rem}\label{results-pc}
  \begin{enumerate}
    \item\label{results-pc0} If $\CP$ is a prepositive cone on $(A,\s)$ and $a \in
      \CP \setminus \{0\}$, then $\CP_F = \CP_F' := \{u \in F \mid ua \in
      \CP\}$. Indeed, we clearly have $\CP_F \subseteq \CP_F'$, and if $u \in
      \CP_F' \setminus \CP_F$, then $-u \in \CP_F$ and thus $ua, -ua \in \CP$,
      contradicting (P5).
    \item A prepositive cone always contains an invertible element, as
      observed in
      \cite[Lemma~3.6]{au20}.
    \item\label{results-pc1} It is possible for a positive cone not to
      contain the element $1$.  This depends on the involution (see
      Proposition~\ref{pos-invo}\eqref{pos-invo1}). For instance, if $a \in \Sym(A, \s) \cap
      A^\x$, then $\CP$ is a positive cone on $(A,\s)$ if and only if $a\CP$ is
      a positive cone on $(A,\Int(a) \circ \s)$ (\cite[Proposition~4.4]{au20}).
      This makes it easy to produce positive cones that contain the
      element $1$ and positive cones that do not.
    \item\label{results-pc2} If $P \in X_F$ there may not be a positive cone over $P$. The set of
      orderings over which there are no positive cones is the set $\Nil[A,\s]$
      of orderings in $X_F$ for which the signature of all hermitian forms is
      zero (\cite[Proposition~6.6]{au20}). The set $\Nil[A,\s]$ is actually a clopen subset of
      $X_F$ (\cite[Corollary~6.5]{au14}).
    \item\label{results-pc4} For $P \in X_F \setminus \Nil[A,\s]$, there are
      exactly two positive cones over $P$. If $\CP$ is one of them, the other
      one is $-\CP$ (\cite[Theorem~7.5]{au20}). This freedom to choose the sign
      comes from the link between signatures of hermitian forms and positive
      cones, and corresponds to the fact that the signature of hermitian forms
      at $P$ is only determined up to sign (\cite[Start of Section~3.3]{au14}).
    \item\label{results-pc3} Let $(D,\vt) \in \{(F,\id), (F(\sqrt{-1}),-),
      ((-1,-1)_F,-)\}$ with $F$ real closed. The only two positive cones on
      $(M_n(D),\vt^t)$ over the unique ordering of $F$ are $\PSD$ and $\NSD$,
      the sets of positive semidefinite, respectively negative semidefinite
      matrices. This follows from the previous item since it can be checked that
      the set of $\PSD$ matrices is a positive cone over the unique ordering of
      $F$ (it is clearly a pre-positive cone; if $\PSD \subsetneq \CP$ with
      $\CP$ positive cone, using the fact that symmetric matrices can be
      diagonalized by congruences, we can assume by (P3) that $\CP$ contains a
      diagonal matrix with at least one negative entry. Using (P3) to only keep
      this negative entry we then obtain a non-zero NSD matrix in $\CP$,
      contradicting (P5)).
    \item\label{results-pc5} For $(D, \vt)$ as in \eqref{results-pc3}, the set of positive
      semidefinite matrices in $M_n(D)$ is equal to $\HS(M_n(D), \vt^t) :=
      \{\vt(M)^tM \mid M \in M_n(D)\}$, the set of hermitian squares in
      $(M_n(D),\vt^t)$. This is due to the principal axis theorem (which also
      holds for quaternions, cf.  \cite[Corollary~6.2]{zhang97}).

      We will also more generally consider hermitian squares in an algebra with
      involution $(A,\s)$ and write
      \[\HS(A,\s) := \{\s(a)a \mid a \in A\}.\]
  \end{enumerate}
\end{rem}

\begin{defi}[{\cite[Definition~1.1]{ps76}}]
  Let $(A,\s)$ be a central simple algebra with involution over $F$, and let $P
  \in X_F$. The involution $\s$ is called positive at $P$ if the form $A\x A\to
  Z(A),\ (x,y)\mapsto \Trd(\s(x)y)$ is positive semidefinite at $P$ (hence
  positive definite at $P$ since it is nonsingular). For more details we refer
  to \cite[Section~4]{au18}. 
\end{defi}

Proposition~\ref{pos-invo} below is mostly a restatement of various results from
\cite{au18, au20}. Its proof is the only one in this paper that uses signatures
of hermitian forms, and is written so that the references can be used without
knowing its definition, but we give a very brief overview of some of the
background in the next remark.

\begin{rem}
  Let $(A,\s)$ be a central simple algebra with involution over $F$.
  A signature of hermitian forms on $(A,\s)$ was introduced in \cite{au14} and studied in
  subsequent papers. The signature of such a hermitian form $h$ at $P \in X_F$ is
  denoted by $\sign^\eta_P h$ ($\eta$ is not very important here, but indicate
  the means by which we fix the sign of the result, since otherwise only its
  absolute value would be well-defined). What is relevant for this paper is the link between
  positive cones, positive involutions, and the value taken by the signature of
  some 1-dimensional diagonal hermitian forms.

  A 1-dimensional diagonal hermitian form over $(A,\s)$ is a map of the form
  \[A \x A \rightarrow A, \quad (x,y) \mapsto \s(x)ay,\]
  for any $a \in \Sym(A,\s)$. We denote it by $\qf{a}_\s$.

  If $P \in X_F \setminus \Nil[A,\s]$, we define two integers:
  \begin{enumerate}
    \item $n_P$, the matrix size of the central simple algebra $A \ox_F F_P$ (it
      is central simple since $P \not \in \Nil[A,\s]$, cf. \cite[p.~499]{au15}).
    \item $m_P := \max\{\sign^\eta_P \qf{a}_\s \mid a \in \Sym(A,\s)\}$, the
      largest signature at $P$ of any diagonal form of dimension 1.
  \end{enumerate}
  It follows from the definition of signatures that $m_P \le n_P$, but we
  actually have $m_P = n_P$ (\cite[Proposition~6.7]{au20}), and the (necessarily invertible) elements $a \in \Sym(A,\s)$ such that
  $\sign^\eta_P \qf{a}_\s = m_P = n_P$ (i.e., who have maximal signature) have
  interesting properties. Let $P \in X_F \setminus \Nil[A,\s]$ and let $a \in
  A^\x$. We have:
  \begin{enumerate}
    \item $\sign^\eta_P \qf{a} = \pm n_P$ if and only if $a$ belongs to one of
      the two positive cones on $(A,\s)$ over $P$ (recall that if $\CP$ is one
      of them, then $-\CP$ is the other), cf. \cite[Proposition~6.7]{au20}.
    \item $|\sign^\eta_P \qf{1}_\s| = n_P$ if and only if the involution $\s$ is
      positive at $P$, cf. \cite[Corollary~4.6]{au18}.
    \item Therefore, considering that the signature at $P$ of $\qf{a}_\s$ is (in
      absolute value) the same as the signature at $P$ of $\qf{1}_{\Int(a) \circ
      \s}$, we have, if $a \in \CP$: $|\sign^\eta_P \qf{1}_{\Int(a) \circ \s}| =
      n_P$, i.e., the involution $\Int(a) \circ \s$ is positive at $P$.
  \end{enumerate}
\end{rem}

\begin{prop}\label{pos-invo}
  Let $F$ be a formally real field.
  \begin{enumerate}
    \item\label{pos-invo1} Let $(A,\s)$ be a central simple algebra with involution over $F$ and
      let $P \in X_F$. Then $\s$ is positive at $P$ if and only if there is a
      positive cone $\CP$ on $(A,\s)$ over $P$ such that $1 \in \CP$.
  \end{enumerate}
  Let $(D,\vt) \in \{(F,\id), (F(\sqrt{-1}),-), ((-1,-1)_F,-)\}$.
  \begin{enumerate}\setcounter{enumi}{1}
    \item The involution $\vt^t$ on $M_n(D)$ is positive at every $P \in X_F$.
    \item Let $a \in \Sym(M_n(D),\vt^t)^\x$ be such that $\Int(a) \circ \vt^t$ is 
      positive at $P$. Then $a$ is a PSD or NSD matrix in $M_n(D)$ with respect
      to the ordering $P$ (and, up to replacing $a$ by $-a$, we can assume that
      $a$ is PSD).
  \end{enumerate}
\end{prop}
\begin{proof}
  \begin{enumerate}
    \item By \cite[Corollary~4.6]{au18} $\s$ is positive at $P$ if and only if
      $|\sign^\eta_P \qf{1}_\s| = n_P$. And by \cite[final line of
      Theorem~7.5]{au20}, $|\sign^\eta_P \qf{1}_\s| = m_P$ if and only if $1$
      belongs to some positive cone over $P$. The link between both statements
      comes from the fact that $m_P = n_P$, cf. \cite[Proposition~6.7]{au20}.
    \item Let $P \in X_F$. The result follows from the first item, since $1 \in
      \PSD$ which is a positive cone on $(M_n(D), \vt^t)$ over $P$ by
      Remark~\ref{results-pc}\eqref{results-pc3}.
  \end{enumerate}
  We make two observations for the proof of the final item:
  \begin{enumerate}[(i)]
    \item\label{o1} In the special case of $(M_n(D),\vt^t)$ with $(D,\vt)$ as indicated in
      the statement, the signature of a hermitian form $\qf{a}_{\vt^t}$ is (up
      to sign) equal to the usual Sylvester signature of the matrix $a$ (this is
      presented in \cite{au18}, page~343: identifying symmetric matrices in
      $M_n(D)$ with hermitian forms, the signature is --up to sign-- the map $\sign_P$ in
      equation~(2.2), and the link with Sylvester signatures in the cases
      considered in this proposition is the final bullet point on that page).
    \item\label{o2} The integer $n_P$ defined in \cite[(6.1)]{au20} is equal to $n$, again
      because of the precise form of the algebra $M_n(D)$ chosen in this proposition.
  \end{enumerate}
  Using these, we have
  \begin{enumerate}\setcounter{enumi}{2}
    \item By \cite[Proposition~4.8]{au18} (where $\s_u$ is defined just before
      \cite[Proposition~4.4]{au18}) we have $\sign^\eta_P \qf{a^{-1}}_{\vt^t} =
      \pm n_P$, which is equal to $\pm n$ by \eqref{o2} above. Since
      $\qf{a^{-1}}_{\vt^t} \cong \qf{a}_{\vt^t}$, we have $\sign^\eta_P
      \qf{a^{-1}}_{\vt^t} = \sign^\eta_P \qf{a}_{\vt^t}$, and thus $\sign^\eta_P
      \qf{a}_{\vt^t} = \pm n$. By \eqref{o1}, it follows that $a$ is PSD if this
      signature is equal to $n$, and NSD if it is equal to $-n$.\qedhere
  \end{enumerate}
\end{proof}

\begin{lemma}\label{rcf-isom}
  \begin{enumerate}
    \item Let $(A,\s)$ be an algebra with involution over $F$, and let $a \in
      A^\x$ be such that $a^{-1} = \s(b)b$ for some $b \in A^\x$. Then $\Int(b)$ is
      an isomorphism of $F$-algebras with involution from $(A, \Int(a) \circ
      \s)$ to $(A,\s)$.
    \item Let $(D,\vt) \in \{(F,\id), (F(\sqrt{-1}),-), ((-1,-1)_F,-)\}$ where
      $F$ is a real closed field, and let $a$ be an invertible PSD matrix in
      $M_n(D)$. Then the $F$-algebras with involution $(M_n(D), \Int(a) \circ
      \vt^t)$ and $(M_n(D),\vt^t)$ are isomorphic.
  \end{enumerate}
\end{lemma}
\begin{proof}
  \begin{enumerate}
    \item We simply have to check that $\Int(b) \circ \Int(a) \circ \s = \s
      \circ \Int(b)$, but this is a direct verification.
    \item The matrix $a^{-1}$ is also PSD and, by the principal axis theorem
      (which holds in $M_n(D)$; see \cite[Corollary~6.2]{zhang97} for the
      quaternion case), there is $b \in M_n(D)$ such that $a^{-1} = \vt(b)^tb$.
      The result now follows from the previous item. \qedhere
  \end{enumerate}
\end{proof}

The next result shows how positive cones give us the control on the
involution that we would like to have in order to get quantifier elimination
(see Remark~\ref{qe-problems}\eqref{problem2}).

\begin{cor}\label{bar-tr}
  Let $(A, \s)$ be a central simple algebra with involution over a real
  closed field $F$, and let $\CP$ be a positive cone on $(A, \s)$ over the unique ordering
  of $F$, such that $1 \in \CP$. Then
  \begin{enumerate}
    \item\label{bar-tr1} There are $(D,\vt) \in \{(F,\id),
      (F(\sqrt{-1}),-), ((-1,-1)_F,-)\}$ and an isomorphism $f : (A,\s)
      \rightarrow (M_n(D), \vt^t)$ of $F$-algebras with involution
      such that $f(\CP) = \PSD$.
    \item\label{bar-tr2} $\CP$ is equal to $\HS(A,\s)$, the set of hermitian
      squares in $(A,\s)$.
  \end{enumerate}
\end{cor}
\begin{proof}
  \begin{enumerate}
    \item Since $F$ is real closed there is an isomorphism $f : A \rightarrow
      M_n(D)$, where $(D,\vt) \in \{(F, \id), (F(\sqrt{-1}), -), ((-1,-1)_F,
      -)\}$. Under this isomorphism, the involution $\s$ becomes $\Int(a) \circ
      \vt^t$, so that $f$ is an isomorphism of algebras with involution from
      $(A,\s)$ to $(M_n(D), \Int(a) \circ \vt^t)$. 

      By Proposition~\ref{pos-invo}(1) and then (3), the involution $\s$ is positive
      at $P$, and thus we can take for $a$ a PSD matrix. We then obtain from
      Lemma~\ref{rcf-isom} an isomorphism of $F$-algebras with involution $g :
      (A,\s) \rightarrow (M_n(D), \vt^t)$.  In particular $g(\CP)$ is a positive
      cone on $(M_n(D), \vt^t)$ over $F$ and $1 \in g(\CP)$, so that $g(\CP) =
      \PSD$ (cf.  Remark~\ref{results-pc}\eqref{results-pc3}).
    \item This is clear since any element of $\PSD$ is of the form $\vt(B)^tB$
      in $M_n(D)$ by the principal axis theorem (see
      \cite[Corollary~6.2]{zhang97} for the quaternion case).\qedhere
  \end{enumerate}
\end{proof}

For the purpose of obtaining a first-order theory, we check that the fact that a
unary relation is a positive cone can be expressed by a first-order formula in
models of $\CSAI_m$:
\begin{lemma}\label{pc-def}
  Let $m \in \N$.
  \begin{enumerate}
    \item There is an $L_R \cup \{\ls, \lCP\}$-formula $\text{PC}_m$ such that for
      every $\CA \models \CSAI_{m}$ and every interpretation $\lCP^\CA$ of $\lCP$
      in $\CA$: $\CA \models \text{PC}_m$ if and only if $\lCP^\CA$ is a positive cone
      on $(A,\s)$.
    \item There is an $L_R \cup \{\ls, \lCP, \lP\}$-formula $\text{PC'}_m$ such that
      for every $\CA \models \CSAI_{m,\of}$ and every interpretation $\lCP^\CA$
      of $\lCP$ in $\CA$: $\CA \models \text{PC'}_m$ if and only if $\lCP^\CA$ is a
      positive cone on $(A,\s)$ over $\lP^\CA$.
  \end{enumerate}
\end{lemma}
\begin{proof}
  We prove the first statement (the second is obtained by taking the conjunction
  of $\text{PC}_m$ and ``$\lP = \lCP_{\lF}$'',  
  where $\lF$ denotes the base field; all of which is first-order in $L_R \cup
  \{\lP, \lCP\}$). We introduce some notation first:

  Let $\CA = (A,\s)$ as algebra with involution over $F$. If $\CP$ is a
  prepositive cone on $(A,\s)$ (over some ordering $P := \CP_F$), and $a \in
  \Sym(A,\s)$, we define
  \[\CP[a] := \{p + \sum_{i=1}^k u_i\s(x_i)ax_i \mid p \in \CP,\ k \in \N,\ u_i
  \in \CP_F,\ x_i \in A\}.\]

    {\bf Claim:} $\CP[a]$ is a prepositive cone over $P$ if and only if $\CP[a] \cap
    -\CP[a] = \{0\}$ (i.e. $\CP[a]$ satisfies (P5)). In this case it is the smallest prepositive cone over
    $P$ containing both $\CP$ and $a$.

    Proof of the Claim: The left to right implication is clear, so we prove
    the other one. It is easily seen that $\CP[a]$ satisfies properties
    (P1), (P2), (P3) from the definition of prepositive cone, and that $P
    \subseteq \CP[a]_F$. If $\CP[a]$ satisfies (P5), it follows that
    $\CP[a]_F$ does not contain any element of $-P \setminus \{0\}$
    (indeed, take any $b \in \CP \setminus \{0\}$; if $u \in \CP[a]_F \cap
    -P \setminus \{0\}$, then $ub \in \CP[a]$ since $u \in \CP[a]_F$, but
    also $-ub \in \CP[a]$ since $-u \in P \subseteq \CP[a]_F$, contradicting (P5)),
    and thus that $\CP[a]_F = P$, so that (P4) holds. The final statement is clear from the
    definition of $\CP[a]$. End of the proof of the Claim.
  \medskip

  The statement ``$\lCP$ is a prepositive cone'' is clearly first-order, so we only need to
  express that it is a maximal prepositive cone. This can be done using the
  Claim above by expressing:
  \[\forall a \in \Sym(A,\s) \setminus \CP \quad \lCP[a] \cap -\lCP[a] \not=
  \{0\},\]
  which itself can be expressed by a first-order formula if ``$z \in \lCP[a]$'' can.
  By definition, $\CP[a]$ is the convex cone over $\CP_F$ generated by the set
  \[\CP \cup \{\s(x)ax \mid x \in A\}.\]
  This is a convex cone in $A$ (with respect to the ordering $\CP_F$), and $\dim_F A
  = m$, so by Carath\'eodory's theorem (which holds for ordered fields, see for
  instance the proof of \cite[Chapter I,
  Theorem~2.3]{barvinok02}), for $z \in A$ we have
  \begin{gather*}
    z \in \CP[a] \\
    \Leftrightarrow \\
    \exists p \in \CP \exists u_1,x_1, \ldots, u_{m+1},x_{m+1} \
    \bigwedge_{i=1}^{m+1} u_i \in \CP_F \wedge z = p + \sum_{i=1}^{m+1}
    u_i\s(x_i)ax_i,
  \end{gather*}
  which is first-order since ``$u \in \CP_F$'' clearly is.
\end{proof}

\subsection{Model-completeness}
We define the theory of ordered central simple algebras with involution of
dimension $m$, and the same theory over a real closed field, in the language $L_\OCSAI := L_\CSAI
\cup \{\lCP\}$, to be:
\[\OCSAI_{m} := \CSAI_{m} \cup \{\lCP \text{ is a positive cone}\}, \text{ and}\]
\[\OCSAI_{m, \rcf} := \CSAI_{m,\rcf} \cup \{\lCP \text{ is a positive cone}\}\]
(with reference to Lemma~\ref{pc-def} for the axiomatization of ``$\lCP$ is a
positive cone'').

\begin{prop}\label{mc-ocsai}
  The theory $\OCSAI_{m,\rcf}$ is model-complete in $L_\CSAI \cup \{\lCP\}$, and
  the theory $\OCSAI_{m,\rcf} \cup \{\lTrd \text{ is the reduced trace}\}$ (cf.
  Lemma~\ref{axiom-Trd}) is model-complete in $L_\CSAI \cup \{\lCP, \lTrd\}$.
\end{prop}
\begin{proof}
  Let $\CM \subseteq \CN$ be two models of $\OCSAI_{m,\rcf}$, and let $F$ and $L$ be
  the base fields of $\CM$ and $\CN$, respectively. Note that $F \subseteq L$,
  either by Theorem~\ref{mod-comp-ec} (since the base field is definable in
  $L_\CSAI$) or more elementarily by Lemma~\ref{ec}(2). 
  By
  Lemma~\ref{inclusion-invo-ec}, we have a diagram as in statement
  \eqref{inclusion-invo2} of this
  Lemma. We know that $\pm\PSD(M_n(D),\vt^t) = \pm\HS(M_n(D), \vt^t)$ are the only positive cones on $(M_n(D),
  \vt^t)$ over the unique ordering of $F$
  (see Remark~\ref{results-pc}\eqref{results-pc5} and \eqref{results-pc3}), and thus that 
  $\pm a \HS(M_n(D), \vt^t)$ are the only positive cones on $(M_n(D), \Int(a) \circ
  \vt^t)$ over the unique ordering of $F$
  (Remark~\ref{results-pc}\eqref{results-pc1}), so that $f(\lCP^\CM) = \ve
  a\HS(M_n(D), \vt^t)$ for
  some $\ve \in \{-1,1\}$. Similarly, $g(\lCP^\CN) = \delta
  a\HS(M_n(E), (\vt')^t)$ for some $\delta \in \{-1,1\}$. Since $f(\lCP^\CM)
  \subseteq g(\lCP^\CN)$, we
  have $\delta = \ve$. Therefore $\lCP^\CM = f^{-1}(a) \HS(\CM, \ls^\CM)$ and
  $\lCP^\CN = f^{-1}(a) \HS(\CN, \ls^\CN)$ are defined by
  the same $L_\CSAI$-formula (with parameter $a$) in $\CM$ and $\CN$. Both
  statements then follow from Lemma~\ref{inclusion-invo-ec}\eqref{inclusion-invo3}.
\end{proof}

\begin{lemma}\label{ext-simple}
  Let $(A,\s)$ be a central simple algebra with involution over $F$, and let $\CP$ be a
  positive cone on $(A,\s)$ over $P \in X_F$. Let $(L,Q)$ be an ordered
  extension of $(F,P)$. Then $A \ox_F L$ is a central simple algebra.
\end{lemma}
\begin{proof}
  If $F = Z(A)$ the result is clear (see for instance \cite[Theorem~1.1(3) and
  (4)]{BOI}), so we can assume that $F \not = Z(A)$, i.e., that $\s$ is of the
  second kind. In particular $Z(A) = F(\sqrt{d})$ for some $d \in F$. By
  hypothesis $P \in X_F \setminus \Nil[A,\s]$ (see \cite[Proposition~6.6]{au20},
  where $\widetilde{X_F} := X_F \setminus \Nil[A,\s]$), so that $d \not \in P$
  by \cite[Proposition~8.4]{au20}. Therefore $\sqrt{d} \not \in L$ (since
  $\sqrt{d}$ is not in the real closure of $(F,P)$) and \[A \ox_F L = A
  \ox_{Z(A)} Z(A) \ox_F L = A \ox_{Z(A)} L[X]/(X^2-d) = A \ox_{Z(A)}
  L(\sqrt{d}),\] which is central simple (see again \cite[Chapter~8,
  Corollary~5.1]{scharlau85}).
\end{proof}

\begin{cor}
  The theory $\OCSAI_{m,\rcf}$ is the model-companion of the theory $\OCSAI_{m}$, and the
  theory $\OCSAI_{m,\rcf} \cup \{\lTrd \text{ is the reduced trace}\}$ is the
  model-com\-pa\-nion of the theory $\OCSAI_{m} \cup \{\lTrd \text{ is the reduced trace}\}$.
\end{cor}
\begin{proof}
  We only prove the second statement, since the first one is similar.

  By Proposition~\ref{mc-ocsai}, it suffices to show that the two theories $\OCSAI_{m} \cup
  \{\lTrd \text{ is the reduced trace}\}$ and $\OCSAI_{m,\rcf} \cup \{\lTrd
  \text{ is the reduced trace}\}$ are cotheories. Let $\CM = (A,\s)$ be a model
  of $\OCSAI_m$, let $F$ 
  be the base field of $\CM$, let $P \in X_F$ be such that $\lCP^\CM$ is
  over $P$, and let $F_P$ be a real closure of $F$ at $P$. By
  Lemma~\ref{ext-simple}, the algebra $(A \ox_F L, \s \ox \id)$ is a central
  simple algebra with involution, and by \cite[Proposition~5.8]{au20}, it is
  equipped with a positive cone containing $\CP \ox 1$. It is thus a model of
  $\OCSAI_{m,\rcf}$ and the inclusion $A \rightarrow A \ox_F L$, $a \mapsto a
  \ox 1$ is a morphism in the language $L_\CSAI \cup \{\lCP, \lTrd\}$, since
  this map preserves the reduced trace by definition of the reduced trace.
\end{proof}

Recall from Remark~\ref{qe-problems}\eqref{problem1} that the theory
$\OCSAI_{m,\rcf}$ does
not have the amalgamation property over substructures (take the PSD matrices for
positive cone on $(M_n(\R), t)$ and on $(M_n((-1,-1)_\R, -^t)$ to turn them into
models of $\OCSAI_{m,\rcf}$), and in particular does not have quantifier
elimination (see \cite[Proposition~3.5.19]{chang-keisler}).

\subsection{Quantifier elimination with an involution of specified type}
\label{with-type}

\begin{defi}\label{theories}
  We introduce two new constant symbols $\la$ and $\lb$ and define the
  following theories in the language $L_\CSAI \cup \{\lCP, \lTrd, \la, \lb\}$,
  where we specify the type of the involution:
  \begin{align*}
    \OCSAI_{m,\rcf}^+ := \OCSAI_{m,\rcf} \cup \{\lTrd \text{ is the reduced
    trace}\}\ \cup \\
       \{\la \in \lCP,\ \lb \text{ is the inverse of } \la\}
  \end{align*}
  and
  \begin{align*}
    \OCSAOI_{m,\rcf}^+ &:= \OCSAI_{m,\rcf}^+ \ \cup \{\ls \text{ is an
    orthogonal involution}\}\\[1ex]
    \OCSASI_{m,\rcf}^+ &:= \OCSAI_{m,\rcf}^+ \ \cup \{\ls \text{ is a
    symplectic involution}\} \\[1ex]
    \OCSAUI_{m,\rcf}^+ &:= \OCSAI_{m,\rcf}^+ \ \cup \{\ls \text{ is a unitary involution}\}
  \end{align*}
  (with reference to Lemma~\ref{axiom-Trd} for the
  axiomatization of the reduced trace).
  Note that every positive cone contains an invertible element by
  \cite[Lemma~3.6]{au20}.
\end{defi}

\begin{rem}\label{qf-centre}
  \begin{enumerate}
    \item We can replace $\la$ and $\lb$ by $1$ in the theories above,
      so that the models will have positive cones that contain $1$ (and will
      correspond to situations where the involution is positive, see
      Proposition~\ref{pos-invo}\eqref{pos-invo1}). In this case there is no
      need to add the new constants $\la$ and $\lb$ to the language.

   \item Recalling Remark~\ref{results-pc}\eqref{results-pc1}, if $a \in
      \CP \cap A^\x$, then $a^{-1} = \s(a^{-1})aa^{-1} \in \CP$ and so $1 \in
      a\CP$ which is a positive cone over $A$ equipped with the ``scaled''
      involution $\Int(a) \circ \s$. It follows from this that the involution
      $\Int(a) \circ \s$ is positive (cf. Proposition~\ref{pos-invo}). The idea
      behind introducing the constants $\la$ and its inverse in the language and
      the theories is to make sure that the substructures are closed under
      $\Int(a)$, so that the involution $\Int(a) \circ \s$ can be defined on them. It allows
      us to work with positive cones (over new, positive, involutions) that
      contain 1. It will be used in the proof of Theorem~\ref{KT-eq}.

    \item\label{qfc} Let $\CM$ be a model of $\OCSAI_{m,\rcf} \cup \{\lTrd \text{ is
      the reduced trace}\}$. Then, as observed in Remark~\ref{equal-field},
      for $a \in \CM$ we have $a \in Z(\CM)$ if and only if $\Trd(a) =
      (\deg \CM)\cdot a$. Since $\deg \CM \le m$ and every non-zero integer
      is invertible in $Z(\CM)$, we obtain that $Z(\CM)$ can be defined by
      the $L_R \cup \{\lTrd\}$-formula
      \[\bigvee_{i=1}^m \lTrd(x) = i \cdot x.\]
  \end{enumerate}
\end{rem}

\begin{thm}\label{KT-eq}
  The theories
  \[\OCSAOI_{m,\rcf}^+, \ \OCSASI_{m,\rcf}^+ \text{ and }
  \OCSAUI_{m,\rcf}^+\]
  each have quantifier elimination in the language $L_\CSAI \cup \{\lTrd, \lCP,
  \la, \lb\}$.
\end{thm}
\begin{proof}
  We only prove it for $\OCSAOI_{m,\rcf}^+$, since the others are similar.

  By Proposition~\ref{mc-ocsai}, the theory $\OCSAOI_{m,\rcf}^+$ is
  model-complete in the language $L_\CSAI \cup \{\lCP, \lTrd\}$,
  so we only need to show that it has the amalgamation property over
  substructures.
  Let $\CM, \CN \models \OCSAOI_{m,\rcf}^+$ and let $\CA$ be a common
  $L_\CSAI \cup \{\lTrd, \lCP, \la, \lb\}$-substructure of $\CM$ and $\CN$. Let
  $\iota_\CM$ and $\iota_\CN$ be the inclusions of $\CA$ in $\CM$ and $\CN$,
  respectively.

  In order to simplify the notation, we write $a$ for the element $\la^\CM = \la^\CA
  = \la^\CN$, $a^{-1}$ for its inverse, $F$ for the 
  base field of $\CM$, and $L$ for the 
  base field of $\CN$.

  We turn $\CM$ and $\CN$ into $L_\CSAI \cup \{\lP, \lTrd,
  \lCP\}$-structures by interpreting $\lP$ in $\CM$ and $\CN$ by the set
  defined by the quantifier-free $L_\CSAI \cup \{\Trd, \lCP\}$-formula
  (cf.  Remark~\ref{results-pc}\eqref{results-pc0}):  
  \begin{equation}\label{Fqf}
    x \in \lF \wedge x\la \in \lCP,
  \end{equation}
  where $x \in \lF$ denotes the formula expressing that $x$ is in the base
  field, which can be expressed by writing that $x$ belongs to the centre
  (this uses $\lTrd$, cf. Remark~\ref{qf-centre}\eqref{qfc}) and that
  $\ls(x) = x$.

  Since formula \eqref{Fqf} is quantifier-free, $\CA$ is an $L_\CSAI \cup
  \{\lP, \lTrd, \lCP\}$-substructure of $\CM$ and $\CN$, i.e., the maps
  $\iota_\CM$ and $\iota_\CN$ are $L_\CSAI \cup \{\lP, \lTrd, \lCP\}$-morphisms.

  Scaling the involutions by $a$, we obtain the following diagram (we only
  indicate the base set, the involution, and the positive cone in each
  entry):
  \[\xymatrix{
     \makebox[.2\textwidth][l]{\hspace{-0em}$\CM' := (\CM, \Int(a) \circ
     \ls^\CM, a \lCP^\CM)$} & &
     \makebox[.2\textwidth][l]{\hspace{-4em}$\CN' := (\CN, \Int(a) \circ
     \ls^\CN, a \lCP^\CN)$} \\
    & \makebox[.2\textwidth][l]{\hspace{-4em}$\CA' := (\CA, \Int(a) \circ
    \ls^\CA, a \lCP^\CA)$} \ar[ul]^{\iota_\CM} \ar[ur]_{\iota_\CN} & \\
  }\]
 Note that $\Int(a)$ is defined on $\CA$ since $a$ and $a^{-1}$
  belongs to $\CA$ (ensuring this is the only reason for having the constant
  $\lb$ in the language).
  It is clear that $\iota_\CM$ and $\iota_\CN$ are still $L_\CSAI \cup \{\lP, \lTrd,
  \lCP\}$-morphisms. Moreover, $a\lCP^\CM$ is a positive cone on
  $(\CM, \Int(a) \circ \ls^\CM)$ by \cite[Proposition~4.4]{au20}, and contains
  $1$ (indeed $a^2 \in a\lCP^\CM$ so that, using that $\ls^\CM(a) = a$ and
  property (P3) of positive cones,
  $\ls^\CM(a^{-1}) a^2 a^{-1} = 1 \in a\lCP^\CM$). Similarly, $1 \in a
  \lCP^\CN$.

  Observe that the involution $\Int(a) \circ \s$ is of the same type as $\s$
  since $a$ is symmetric (which follows from $a \in \CP$), by
  \cite[Proposition~2.7(3)]{BOI}. Therefore, by Corollary~\ref{bar-tr},
  we have two $L_\CSAI \cup \{\lP, \lTrd, \lCP\}$-isomorphisms (whose images are
  $M_n(F)$ and $M_n(E)$ since the involution is orthogonal, see the end of Section~\ref{awi}): 
  \[\phi : \CM' \rightarrow (M_n(F), {}^t, \PSD) \text{ and } \psi : \CN'
  \rightarrow (M_n(L), {}^t, \PSD)\]
  where $\lTrd$ is interpreted by
  the reduced trace in $M_n(F)$ and $M_n(L)$, and the maps $\phi$ and $\psi$
  respect $\lTrd$ by Lemma~\ref{trace-morphism}.

  By Lemma~\ref{amalg-arg}, we can amalgamate this diagram using two
  $L_{\CSAI} \cup \{\lP, \lTrd\}$-morphisms $\lambda : M_n(F) \rightarrow
  M_n(K)$ and $\mu : M_n(L) \rightarrow M_n(K)$. Since $\PSD$ is the set of
  hermitian squares in all three structures
  (Corollary~\ref{bar-tr}\eqref{bar-tr2}), $\lambda$ and $\mu$ respect $\PSD$,
  so that the following diagram consists of $L_\CSAI \cup \{\lP, \lTrd,
  \lCP\}$-morphisms:
  \[\xymatrix{
    & (M_n(K), {}^t, \PSD) & \\
    (M_n(F),{}^t, \PSD) \ar[ur]^\lambda & & (M_n(L), {}^t, \PSD) \ar[ul]_\mu \\
    \CM' \ar[u]^\phi & & \CN' \ar[u]_\psi \\
    & \CA' \ar[ul]^{\iota_\CM} \ar[ur]_{\iota_\CN} & \\
  }\]

  We now ``scale back'' and check that the morphisms (which are the same
  at the level of the elements) are still morphisms in the language $L_\CSAI
  \cup \{\lP, \lTrd, \lCP\}$:
  \[\xymatrix{
    & \makebox[.2\textwidth][l]{\hspace{-5em}$(M_n(K), \Int(\lambda \circ \phi(a^{-1})) \circ {}^t,
      \lambda \circ \phi(a^{-1}) \PSD)$} & \\
    \makebox[.2\textwidth][l]{\hspace{-4.5em}$(M_n(F), \Int(\phi(a^{-1})) \circ {}^t,
    \phi(a^{-1}) \PSD)$} \ar[ur]^\lambda & &
    \makebox[.2\textwidth][l]{\hspace{-5.5em}$(M_n(L), \Int(\psi(a^{-1})) \circ {}^t,
      \psi(a^{-1}) \PSD)$} \ar[ul]_\mu \\
    \CM \ar[u]^\phi & & \CN \ar[u]_\psi \\
    & \CA \ar[ul]^{\iota_\CM} \ar[ur]_{\iota_\CN} & \\
  }\]
  The morphisms fix the elements of $F$ pointwise, so are clearly morphisms for
  $\{\lP\}$. We check that the maps are $\{\ls, \lCP\}$-morphisms:
  \begin{itemize}
    \item The map $\phi$. For the positive cones, we want $\phi(\lCP^\CM)
      \subseteq \phi(a^{-1}) \PSD$, which is clear since $\phi(a\lCP^\CM)
      \subseteq \PSD$ by construction of $\phi$. For the involutions, we want
      $\phi \circ \s(x) = \Int(\phi(a^{-1})) \circ {}^t \circ \phi(x)$ for every
      $x \in \CM$. We successively have, for every $x' \in \CM$:
      \begin{align*}
        \phi \circ \Int(a) &\circ \ls^\CM(x') = {}^t \circ \phi(x')  \quad \text{(by
        definition of $\phi$)} \\
        &\Leftrightarrow \phi(a\ls^\CM(x')a^{-1}) = \phi(x')^t  \\
        &\Leftrightarrow  \phi(\ls^\CM(a^{-1}x'a)) = \phi(x')^t  \quad
        \text{(using that $\ls^\CM(a) = a$)} \\
        &\Leftrightarrow  \phi \circ \ls^\CM(x) = \phi(axa^{-1})^t  \quad \text{(where
         $x:=a^{-1}x'a$)} \\
        &\Leftrightarrow  \phi \circ \ls^\CM(x) = \phi(a^{-1})^t \phi(x)^t \phi(a)^t  \\
        &\Leftrightarrow  \phi \circ \ls^\CM(x) = \Int(\phi(a^{-1})^t) \circ {}^t \circ \phi(x)
      \end{align*}
      and the result follows since $\phi(a)^t = \phi(a)$ (indeed, using that
      $\phi \circ \Int(a) \circ \ls^\CM = {}^t \circ \phi$ by definition of $\phi$,
      we obtain $\phi(a) = \phi(a\ls^\CM(a)a^{-1}) = \phi(a)^t$).
    \item The map $\psi$: It is the exact same argument as for $\phi$.
    \item The map $\lambda$. It is clear for the positive cones. For the
      involutions, we want $\lambda \circ \Int(\phi(a^{-1})) \circ {}^t=
      \Int(\lambda \circ \phi(a^{-1})) \circ {}^t \circ \lambda$. Let $X \in
      M_n(F)$. Then, using that ${}^t \circ \lambda = \lambda \circ {}^t$ (by
      definition of $\lambda$), we
      get
      \begin{align*}
        \Int(\lambda \circ \phi(a^{-1})) \circ {}^t \circ \lambda(X) &=
          \lambda \circ \phi(a^{-1}) \lambda(X)^t \lambda \circ \phi(a) \\
          &= \lambda(\phi(a^{-1}) X^t \phi(a)) \\
          &= \lambda \circ \Int(\phi(a^{-1})) \circ {}^t(X).
      \end{align*}
    \item The map $\mu$. Recall that by choice of $\lambda$ and $\mu$ we have
      $\mu \circ \psi(a) = \lambda \circ \phi(a)$ and thus $\mu \circ
      \psi(a^{-1}) = \lambda \circ \phi(a^{-1})$. For the positive cones, and
      since $\mu(\PSD) \subseteq \PSD$, we have $\mu(\psi(a^{-1}) \PSD) = \mu
      \circ \psi(a^{-1}) \mu(\PSD) \subseteq \lambda \circ \phi(a^{-1}) \PSD$.
      For the involutions, we want $\mu \circ \Int(\psi(a^{-1})) \circ {}^t =
      \Int(\lambda \circ \phi(a^{-1})) \circ {}^t \circ \mu$, i.e., $\mu \circ
      \Int(\psi(a^{-1})) \circ {}^t = \Int(\mu \circ \psi(a^{-1})) \circ {}^t
      \circ \mu$. Using that ${}^t \circ \mu = \mu \circ {}^t$ (by definition of
      $\mu$), we get
      \begin{align*}
        \Int(\mu \circ \psi(a^{-1})) \circ {}^t \circ \mu(X) &=
          \mu \circ \psi(a^{-1}) \mu(X)^t \mu \circ \psi(a) \\
          &= \mu(\psi(a^{-1}) X^t \psi(a)) \\
          &= \mu \circ \Int(\psi(a^{-1})) \circ {}^t(X). \qedhere
      \end{align*}
  \end{itemize}
\end{proof}

\begin{cor}
  Each of the final three theories from Definition~\ref{theories} is the
  model-companion of the same theory where we do not specify that 
  the base field is real closed.
\end{cor}
\begin{proof}
  We prove it for $T := \OCSASI_{m,\rcf}^+$, the argument is the same for the
  other two. Let $T_0$ be this theory without specifying that 
  the base field is real closed. We know that $T$ is model-complete by
  Theorem~\ref{KT-eq}, so we only have to show that $T$ and $T_0$ are
  cotheories: Let $(A,\s)$ be a central simple algebra with involution over $F$
  that is a model of $T_0$, so that there is a positive cone $\CP$ on $(A,\s)$.
  Let $L$ be a real closure of $F$ at the ordering $\CP_F$. Then $(A \ox_F L, \s
  \ox \id)$ is a central simple algebra with involution over $L$ (see
  Lemma~\ref{ext-simple}), is equipped with a positive cone $\CQ$ containing
  $\CP \ox 1$ (\cite[Proposition~5.8]{au20}), and is thus a model of $T$ that
  contains$(A,\s,\CP, \Trd, a, a^{-1})$ as an $L_\CSAI \cup \{\lCP, \lTrd,
  \la, \lb\}$-substructure (the reduced trace of an element of $A$ remains the same
  after scalar extension, by definition of the reduced trace).
\end{proof}

\section{Correspondence between positive cones and morphisms}

The model-completeness of $\OCSAI_{m,\rcf}$ makes it interesting to point out that
positive cones on central simple algebras with involution are in bijection with morphisms of
$L_\CSAI$-structures into models of $\OCSAI_{m,\rcf}$.

\begin{lemma}\label{nil}
  Let $(A,\s)$ be a central simple algebra with involution over
  $F$ and let $P \in X_F$. Let $F_P$ be a real closure of $F$ at $P \in X_F$.
  Then:
  \begin{enumerate}
    \item If $\s$ is symplectic, $\Nil[A,\s] = \{P \mid A \ox_F F_P \cong
      M_n(F_P)\}$.
    \item If $\s$ is orthogonal, $\Nil[A,\s] = \{P \mid A \ox_F F_P \cong
      M_n((-1,-1)_{F_P})\}$.
    \item If $\s$ is unitary,
      \begin{align*}
        \Nil[A,\s] = \{P \mid\ & A \ox_F F_P \cong M_n(F_P) \x M_n(F_P) \text{ or
        }\\
        & A \ox_F F_P \cong M_{n/2}((-1,-1)_{F_P}) \ox M_{n/2}((-1,-1)_{F_P})\}.
      \end{align*}
  \end{enumerate}
\end{lemma}
\begin{proof}
  The original definition of $\Nil[A,\s]$ (\cite[Definition~3.7]{au14}) is the
  list given in the Lemma (where the case $M_{n/2}((-1,-1)_{F_P}) \ox
  M_{n/2}((-1,-1)_{F_P})$ for $\s$ unitary was missing, an omission corrected in
  \cite[p.~499]{au15}). The fact that this original definition coincides with the one
  used in this paper (the set of orderings at which the signatures of all
  hermitian forms are zero) is \cite[Theorem~6.1]{au14}.
\end{proof}

\begin{prop}\label{one-direction}
  Let $(A,\s) \models \CSAI_{m}$, $(B,\tau) \models \OCSAI_{m,\rcf}$ and let $f : A
  \rightarrow B$ be a morphism of $L_\CSAI$-structures (i.e., $f$ is a morphism
  of rings with involution).
  Then $f^{-1}(\lCP^B)$ is a positive cone on $(A,\s)$ over $P$, where $P$
  denotes the preimage by $f$ of the unique ordering on 
  the base field of $(B,\tau)$ (recall that $f$ sends the base field of $(A,\s)$
  into the base field of $(B,\tau)$ by Lemma~\ref{ec}).
\end{prop}
\begin{proof}
  To simplify notation, we write 
  $F$ for the base field of $(A,\s)$,
  $L$ for the base field of $(B,\tau)$, $\CQ := \lCP^B$, and $Q$ for the
  unique ordering of $L$, 
  so that $f|_F$ is a morphism of ordered fields from $(F,P)$ to $(L,Q)$.

  We first prove that $P \not \in \Nil[A,\s]$. For this we
  consider two cases:

  (1) If $\s$ is of the first kind, i.e., $F = Z(A)$. 
  We extend $f$ to $f' : A \ox_F L \rightarrow B$, $f'(a \ox
  \ell) = f(a) \ell$ (using the action of $F$ on $L$ via $f$ for the tensor
  product). Since $F = Z(A)$, $A \ox_F L$ is simple, and $f'$ is
  injective and therefore bijective (recall that $\dim_L A \ox_F L = \dim_F A = 
  \dim_L B$). Since $f'$ is easily seen to respect the involutions, $f'$ is an
  isomorphism of algebras with involution from $(A \ox_F L, \s \ox \id)$ to
  $(B, \tau)$.
  By hypothesis $Q \not \in \Nil[B,\tau]$, so $B \cong A \ox_F L$ is
  isomorphic to $M_m(L)$ or $M_{m/4}((-1,-1)_L)$ by Lemma~\ref{nil}. Denote by
  $\iota$ the inclusion of $F$ in $F_P$ and let $\lambda : F_P \rightarrow L$ be
  such that $f|_F = \lambda \circ \iota$. Then $A \ox_F L \cong (A \ox_\iota
  F_P) \ox_\lambda L$, so that $A \ox_\iota F_P$ must be isomorphic to either
  $M_m(F_P)$ or $M_{m/4}((-1,-1)_{F_P})$, and thus $P \not \in \Nil[A,\s]$ by
  Lemma~\ref{nil}.

  (2) If $\s$ is of the second kind, i.e., $Z(A) = F(\sqrt{d})$ for some $d \in
  F$. We can choose $d$ such that $\s(\sqrt{d})= -\sqrt{d}$, and thus $\tau(f(\sqrt{d})) =
  -f(\sqrt{d})$. Assume that $P \in \Nil[A,\s]$, so that $d \in P$
  (\cite[Proposition~8.4]{au20}). Since $\sqrt{d}$ is invertible in $A$,
  $f(\sqrt{d})$ is invertible in $B$. Moreover $f(\sqrt{d})^2 = f(d) \in L$.
  Since $f$ is a morphism of ordered fields from $(F,P)$ to the real closed field $(L,Q)$, we have
  $f(\sqrt{d}) \in Q$ and there is $\alpha \in L$ such that
  $\alpha^2 = f(d)$. Therefore, in the field $L(f(\sqrt{d}))$ the elements $\alpha,
  -\alpha, f(\sqrt{d})$ are roots of $X^2-f(d)$, so that $f(\sqrt{d}) = \pm\alpha
  \in L$ and thus $\tau(f(\sqrt{d})) = f(\sqrt{d})$, contradiction.\medskip

  Since $P \not \in \Nil[A,\s]$, there is a positive cone $\CP$ on $(A,\s)$ over
  $P$, and $A \ox_F L$ is simple by Lemma~\ref{ext-simple}. Going back to the
  argument presented in (1) above, the map $f'$ is then an isomorphism of
  algebras with involutions, even if $\s$ is of the second kind.

  By \cite[Proposition~5.8]{au20} there is a positive cone on $(A \ox_F
  L, \s \ox \id)$ over $Q$ containing $\CP \ox 1$ (the Proposition is written
  for an inclusion of fields, but applies also here: replace $A$ by $A \ox_F
  f(F)$, then use the inclusion from $f(F)$ into $L$), and thus there is
  a positive cone $\CS$ on $(B,\tau)$ over $Q$ such that $f'(\CP \ox 1) \subseteq
  \CS$. Since $\CQ$ and $-\CQ$ are the only positive cones on $(B,\tau)$ over
  $Q$, up to replacing $\CP$ by $-\CP$ we must have $f'(\CP \ox 1) \subseteq
  \CQ$. It follows that $\CP \subseteq f^{-1}(\CQ)$, and thus that $\CP =
  f^{-1}(\CQ)$ since $\CP$ is a positive cone and $f^{-1}(\CQ)$ is easily seen
  to be a prepositive cone ((P1), (P2), (P3) are clear, (P5) holds since $f$ is
  injective because $A$ is simple, and for (P4) it suffices to check that
  $P \subseteq (f^{-1}(\CQ))_F$, the other inclusion following from (P5)).
\end{proof}

Conversely, every positive cone on $(A,\s)$ can be obtained in this way:
\begin{prop}\label{other-direction}
  Let $(A,\s) \models \CSAI_{m}$, let $F$ be the base field of $(A,\s)$, and let
  $\CP$ be a positive cone on $(A,\s)$ over $P \in X_F$. 
  Then there is a model $(B,\tau)$ of $\OCSAI_{m,\rcf}$ and a morphism of
  $L_\CSAI$-structures $f : (A,\s) \rightarrow (B,\tau)$ such that $\CP =
  f^{-1}(\lCP^B)$.
\end{prop}
\begin{proof}
  Take for $f$ the canonical map $A \rightarrow A \ox_F
  F_P$, $a \mapsto a \ox 1$, where $F_P$ is a real closure of $F$ at $P$.
  The $F_P$-algebra $A \ox_F F_P$ is central simple by Lemma~\ref{ext-simple}
  and by \cite[Proposition~5.9]{au20} there is a positive cone $\CQ$ on $(A
  \ox_F F_P, \s \ox \id)$ over the unique ordering of $F_P$, so that $f(\CP)
  \subseteq \CQ$. Let $\CB$ be the natural $L_{\CSAI}$-structure on $(A \ox_F
  F_P, \s \ox \id)$, in which we also interpret the symbol $\lCP$ by $\CQ$, thus
  turning $(B,\tau)$ into a model of $\OCSAI_{m,\rcf}$.

  We check that $\CP = f^{-1}(\CQ)$. By construction, we have $\CP \subseteq
  f^{-1}(\CQ)$. It is easy to see that $f^{-1}(\CQ)$ is a prepositive cone on
  $(A,\s)$ over $P$. Since $\CP$ is a positive cone on $(A,\s)$ over $P$, we
  must have $\CP = f^{-1}(\CQ)$.
\end{proof}

\section*{Acknowledgement}

I would like to thank the referees for their careful reading of the paper
and for their many valuable suggestions, which led to substantial
improvements.

\bibliographystyle{plain}

\def\cprime{$'$}

\end{document}